\documentclass[11pt,a4paper]{article}
\usepackage{amssymb,amsthm,amsmath,amsfonts}
\theoremstyle{plain}
\newtheorem{theorem}{Theorem}[section]
\newtheorem{proposition}[theorem]{Proposition}
\newtheorem{corollary}[theorem]{Corollary}

\newtheorem{lemma}[theorem]{Lemma}

\theoremstyle{definition}

\theoremstyle{remark}

\newtheorem{remark}[theorem]{Remark}

\numberwithin{equation}{section}

\def\Ch{\text{Ch}}

\def\Im{\text{Im}}
\def\Re{\text{Re}}

\renewcommand{\and}{\text{and}}

\begin{document}

\title{Various $3 \times 3$ Nonnegative Matrices
        with \\  Prescribed Eigenvalues and Diagonal Entries}

\author{Jin Ok Hwang\footnote{aldif89@korea.ac.kr} and Donggyun Kim\footnote{kim.donggyun@gmail.com} \\
{\small Department of Mathematics, Korea University}}

\date{} 

\maketitle

\begin{abstract}
 In this paper, we answer the various forms of nonnegative inverse eigenvalue problems with prescribed diagonal entries for order three: real or complex general matrices, symmetric stochastic matrices, and real or complex doubly stochastic matrices. We include the known cases, the symmetric matrices and real or complex stochastic matrices, to compare the other results and for completeness. In addition, for a given list of eigenvalues, we compute the exact range for the largest value of the diagonal entries of the various nonnegative matrices.
\end{abstract}


\tableofcontents

%

\section{Introduction}

The \emph{nonnegative inverse eigenvalue problem} (NIEP) is to determine
necessary and sufficient conditions for a list of real or complex numbers
to be the eigenvalues of some nonnegative matrix of real numbers.
For a list $\Lambda$ of numbers, if there exists a nonnegative matrix $A$ with eigenvalues $\Lambda$, we say that $\Lambda$ is \emph{realizable} and that $A$ \emph{realizes} $\Lambda$.
The nonnegative matrix could be of various forms: general, symmetric, stochastic, symmetric stochastic, and doubly stochastic nonnegative matrices. By a complex stochastic NIEP, we mean the NIEP of finding a stochastic nonnegative matrix over real numbers with complex eigenvalues.

The origin of the NIEP goes back to Kolmogorov \cite{Kolmogoroff1937} in 1937, who asked when a given complex number is an eigenvalue of some nonnegative matrix. Suleimanova \cite{Suleimanova1945} extended Kolmogorov's question to the NIEP form in 1949. Since then, vast literature has been devoted to its study, but we are far from a satisfactory solution.

Let $n$ be the cardinality of list $\Lambda$. For the case $n=3$, the general and symmetric NIEPs for a list of real numbers are straightforward, and the general NIEP for a list of complex numbers was proved by Loewy and London \cite{LoewyLondon1978} in 1978. Perfect \cite{Perfect55} in 1955 solved the stochastic NIEP for a list of real numbers and by a slight extension, Soto, Salas, and Manzaneda \cite{Soto2010} in 2010 solved this problem for a complex list. In 1965, Perfect and Minsky \cite{PerfectMirsky65} solved the symmetric stochastic and doubly stochastic NIEPs for lists of real and complex numbers.

For the case $n=4$, in \cite{LoewyLondon1978} in 1978, Loewy and London also gave a solution for a general NIEP for a list of real numbers. For a general NIEP for a list of complex numbers, a solution in terms of the power sums of the elements of $\Lambda$ appeared in the Ph.D. thesis of Meehan \cite{Meehan1998} in 1998 and a solution in terms of the coefficients of the characteristic polynomial of a nonnegative matrix was published by Torre-Mayo, Abril-Raymundo, Alarcia-Est\'evez, Mariju\'an, and Pisonero \cite{Torre-Mayo-ed2007} in 2007. Wuwen \cite{Wuwen1996} presented a solution for the symmetric NIEP in 1996. For other NIEPs, only partial answers are known, see for example \cite{Mourad2006} and \cite{Mourad2009}. For the case $n \ge 5$,  the NIEP is unsolved in the general case and we refer to \cite{Soto2013}.

We may ask further that for two given numerical lists $\Lambda$ and $\Omega$, what the necessary and sufficient conditions  are for $\Lambda$ to be the eigenvalues and $\Omega$ the diagonal entries of some nonnegative matrix  \cite{LaffeySmigoc2007}. We call this kind of question a NIEP with prescribed diagonal entries.

For the case $n=3$, the symmetric NIEP with prescribed diagonal entries was proved by Fieldler \cite{Fiedler1974} in 1974. The real stochastic NIEP with prescribed diagonal entries was proved by Perfect \cite{Perfect55} in 1955 and by a slight extension, the complex stochastic NIEP with prescribed diagonal entries was proved by Soto, Salas, and Manzaneda \cite{Soto2010} in 2010.

In this paper, we answer the various forms of nonnegative inverse eigenvalue problems with prescribed diagonal entries for $n=3$: real or complex general matrices (section \ref{sec:General}), symmetric stochastic matrices (section \ref{sec:Sym-stocha}), and real or complex doubly stochastic matrices (section \ref{sec:Doubly-stochas}). We include the known cases, the symmetric matrices (section \ref{sec:Symmetric}) and real or complex stochastic matrices (section \ref{sec:Stochastic}), to compare the other results and for completeness. In addition, for a given list of eigenvalues, we compute the exact range for the largest value of the
diagonal entries of the various nonnegative matrices.

%

\section{Preliminary}

We say that a matrix is \emph{nonnegative} if all its entries are nonnegative real numbers.
A basic theorem in the theory of nonnegative matrices, due to Perron and Frobenius, see for example \cite{Minc1988}, states that the eigenvalues of a nonnegative matrix $A$ have to contain a nonnegative number that is greater than or equal to the absolute value of any member of the eigenvalues of $A$. Since for any nonnegative matrix $A$ with eigenvalues $\Lambda=\{\lambda_1, \lambda_2, \dots, \lambda_n \}$ the trace of $A^k,\ k=1,2,\dots $ is nonnegative, the following condition must hold.
\[  s_k(\Lambda) = \lambda_1^k + \lambda_2^k+ \cdots+ \lambda_n^k \ge 0.\]
Loewy and London \cite{LoewyLondon1978} and Johnson \cite{Johnson1981} independently proved the following result, known as the JLL inequalities:
\[ n^{k-1} s_{km}(\Lambda) \ge \ s_m^k(\Lambda)\]
for all positive integers $k$ and $m$. Since the characteristic polynomial of a nonnegative matrix is a polynomial over real numbers, the eigenvalues are closed under complex conjugation. These four statements are necessary conditions for the list $\Lambda$ of complex numbers to be the eigenvalues of a nonnegative matrix. The necessary conditions are sufficient only when $\Lambda$ has at most three elements \cite{LoewyLondon1978}.

Before we answer the various forms of NIEP with prescribed diagonal entries for $n=3$, we present the cases for $n=2$ which are trivial, but give a hint for the case $n=3$.

Let $A $ be a $2 \times 2$ nonnegative matrix over the field of real numbers.
By the Perron-Frobenius theorem, $A$ has a nonnegative eigenvalue, and this implies that the other eigenvalue of $A$ is also real.

\begin{proposition}
  Let $\Lambda=\{\lambda_1, \lambda_2\}$ and  $\Omega=\{\omega_1, \omega_2\}$ be lists of real numbers
with $\lambda_1 \ge \lambda_2$ and $\omega_1 \ge \omega_2$.
\begin{enumerate}
\item There is a general, symmetric, or stochastic nonnegative matrix with eigenvalues $\Lambda$ and
diagonal entries $\Omega$ if and only if the lists $\Lambda$ and $\Omega$ satisfy
    \begin{description}
        \item[](i) $\omega_i \ge 0$ for $i= 1 \hbox{\ and\ } 2$,
        \item[](ii) $\omega_1+ \omega_2 = \lambda_1 + \lambda_2$,
        \item[](iii) $\omega_1 \omega_2 \ge \lambda_1 \lambda_2$.
    \end{description}
Moreover, in this case, the exact range for $\omega_1$ is
        \[
        \frac{1}{2}(\lambda_1 +\lambda_2) \le \omega_1 \le \min\{\lambda_1 +\lambda_2, \lambda_1\}
        \]
        and we have
        \[
        \omega_2=\lambda_1 +\lambda_2 -\omega_1.
        \]

\item There is a symmetric stochastic nonnegative matrix with eigenvalues $\Lambda$ and
diagonal entries $\Omega$ if and only if the lists $\Lambda$ and $\Omega$ satisfy
    \begin{description}
        \item[](i) $\omega_i \ge 0$ for $i= 1 \hbox{\ and\ } 2$,
        \item[](ii) $\lambda_1 + \lambda_2 = \omega_1+ \omega_2$,
        \item[](iii) $\omega_1 =  \omega_2$.
    \end{description}
Moreover, in this case, we have
        \[
        \omega_1=\omega_2=\frac{1}{2}(\lambda_1 +\lambda_2).
        \]

\item The list $\Lambda$ consists of the eigenvalues of a general, symmetric, stochastic or symmetric stochastic nonnegative matrix if and only
if the following holds:
    \[ \lambda_1 + \lambda_2\ge 0.
    \]
\end{enumerate}
\end{proposition}

For $\Lambda$ and $\Omega$ satisfying the conditions in the first statement, we find
the following general, symmetric and stochastic nonnegative matrices with eigenvalues $\Lambda$ and diagonal entries $\Omega$
\begin{gather*}
       \left(\begin{array}{cccccccc}
        \omega_1 & \omega_1 \omega_2-\lambda_1 \lambda_2  \\ 1 & \omega_2
        \end{array}\right),\quad
        \left(\begin{array}{cccccccc}
        \omega_1 & \sqrt{\omega_1 \omega_2-\lambda_1 \lambda_2} \\
        \sqrt{\omega_1 \omega_2-\lambda_1 \lambda_2}  & \omega_2
        \end{array}\right), \\
      \hbox{and\ \ } \left(\begin{array}{cccccccc}
        \omega_1 & \lambda_1-\omega_1 \\
        \lambda_1 - \omega_2 & \omega_2
        \end{array}\right),
\end{gather*}
for each condition, respectively.

For $\Lambda$ and $\Omega$ satisfying the conditions in the second statement,
we find the following symmetric stochastic nonnegative matrix
\begin{equation*}
        \left(\begin{array}{cccccccc}
        \frac{1}{2}(\lambda_1 + \lambda_2) & \frac{1}{2}(\lambda_1 -\lambda_2) \\
        \frac{1}{2}(\lambda_1 -\lambda_2)  & \frac{1}{2}(\lambda_1 + \lambda_2)
        \end{array}\right),
\end{equation*}
with eigenvalues $\Lambda$ and diagonal entries $\Omega$. Note that this matrix
is symmetric doubly stochastic.

The proof is straightforward and one may refer to the proofs of the $3 \times 3$ matrix cases in the following sections.

%
%

\section{General $3 \times 3$ nonnegative matrices} \label{sec:General}

We discuss a general $3 \times 3$ nonnegative matrix case with prescribed real or complex eigenvalues and diagonal entries. Consider first the case of real eigenvalues.

Statement \ref{enum:rg4} in the following Theorem \ref{the:rg} was proved by Loewy and London \cite{LoewyLondon1978}, but here we prove statement \ref{enum:rg4} as a corollary of statements \ref{enum:rg1} and \ref{enum:rg3} in  Theorem \ref{the:rg}. This kind of pattern will appear repeatedly.

\begin{theorem}[real, general case]\label{the:rg}
  Let $\Lambda=\{\lambda_1, \lambda_2, \lambda_3\}$ and $\Omega=\{\omega_1, \omega_2, \omega_3\}$
be lists of real numbers with $\lambda_1 \ge \lambda_2 \ge\lambda_3$ and $\omega_1 \ge\omega_2\ge \omega_3$.

\begin{enumerate}
\item\label{enum:rg1} There is a nonnegative matrix with eigenvalues $\Lambda$ and diagonal entries $\Omega$ if and only if the lists
$\Lambda$ and $\Omega$ satisfy
 \begin{description}
    \item[] (i) $\omega_i \ge 0$ for $i= 1, 2 \hbox{\ and\ } 3$,
    \item[] (ii) $\lambda_1 \geq \omega_1 \ge \lambda_2$,
    \item[] (iii) $\omega_1+~\omega_2 +\omega_3 = \lambda_1 + \lambda_2+ \lambda_3$,
    \item[] (iv) $\omega_1 \omega_2+\omega_1 \omega_3+\omega_2 \omega_3 \ge \lambda_1 \lambda_2+\lambda_1 \lambda_3 +\lambda_2 \lambda_3$.
 \end{description}

\item\label{enum:rg3} If $\Lambda$ and $\Omega$ satisfy the conditions of statement \ref{enum:rg1},
then the exact range for $\omega_1$ is
        \begin{equation}\label{eq:rgw1}
        \max\{\frac{1}{3}(\lambda_1 +\lambda_2+\lambda_3), \lambda_2\} \le \omega_1 \le \min\{\lambda_1 +\lambda_2+\lambda_3, \lambda_1\}.
        \end{equation}
 For the values of $\omega_2$ and $\omega_3$, we may take
  \begin{equation}\label{eq:rgw2}
        \omega_2=\omega_3=\frac{1}{2}(\lambda_1 +\lambda_2+\lambda_3 -\omega_1).
        \end{equation}

\item\label{enum:rg4}\cite{LoewyLondon1978} The list $\Lambda$ is realizable as the eigenvalues of a nonnegative matrix if and only if the following hold:
    \begin{description}
        \item[] (i) $\lambda_1 +\lambda_3 \ge 0$,
        \item[] (ii) $\lambda_1 + \lambda_2 +\lambda_3 \ge 0$.
    \end{description}

\end{enumerate}
\end{theorem}

\label{enum:rg2} For $\Lambda$ and $\Omega$ satisfying the conditions of statement \ref{enum:rg1}, we find the following nonnegative matrix
  \[  \left(\begin{array}{cccccccc}
    \omega_1 & 0 & (\lambda_1-\omega_1)(\lambda_2-\omega_1)
         (\lambda_3-\omega_1)\\
     1 & \omega_2 & \omega_1 \omega_2+\omega_1 \omega_3+\omega_2 \omega_3 - \lambda_1 \lambda_2-\lambda_1 \lambda_3 -\lambda_2 \lambda_3   \\
     0 & 1 & \omega_3
            \end{array}\right)
   \]
 with eigenvalues $\Lambda$ and diagonal entries $\Omega$. In this  and hereafter nonnegative matrices, we choose that the diagonal entries are in descending order. We may also choose the diagonal entries in ascending order by applying a suitable permutation similarity.

\begin{proof}
\begin{enumerate}
\item
 Suppose that $\Lambda$ is realizable as the eigenvalues of a nonnegative matrix $A$
with diagonal entries $\Omega$. Then, the elements $\omega_1$, $\omega_2$, and $\omega_3$
 are nonnegative, and therefore, we have condition $(i)$.
 The characteristic polynomial $\Ch_1(\lambda)$ of $A$, using the eigenvalues $\Lambda$, is
 \[ \Ch_1(\lambda) = \lambda_1 \lambda_2 \lambda_3 -(\lambda_1 \lambda_2+\lambda_1 \lambda_3
+\lambda_2 \lambda_3)\lambda+ (\lambda_1 + \lambda_2 + \lambda_3)\lambda^2-\lambda^3. \]

  We may write the nonnegative matrix $A$ as
  \[ A=\left(\begin{array}{cccccccc}
    \omega_1 & a_{12} & a_{13} \\
    a_{21}  & \omega_2 & a_{23} \\
    a_{31} & a_{32} & \omega_3
             \end{array}\right),
  \]
 where all elements $a_{ij}$ are nonnegative.
 Hence, we have another form of the characteristic polynomial $\Ch_2(\lambda)$ of $A$,
  \begin{multline*}
   \Ch_2(\lambda) = a_{12} a_{23} a_{31} + a_{13} a_{21} a_{32} - a_{23} a_{32} \omega_1 - a_{13} a_{31} \omega_2 - a_{12} a_{21} \omega_3 + \omega_1 \omega_2 \omega_3 \\
   - (\omega_1 \omega_2+\omega_1 \omega_3+\omega_2 \omega_3 - a_{12} a_{21} - a_{23} a_{32}- a_{13} a_{31}) \lambda \\
   + (\omega_1 + \omega_2  + \omega_3 )\lambda^2 - \lambda^3.
  \end{multline*}
Then, the two characteristic polynomials $\Ch_1(\lambda)$ and $\Ch_2(\lambda)$ must be identical. Comparing the coefficients of $\lambda^2$, we have
\[ \omega_1+ \omega_2+ \omega_3 = \lambda_1 + \lambda_2 +\lambda_3,\]
which is condition $(iii).$
Comparing the coefficients of $\lambda$, we have
\[ \omega_1 \omega_2+\omega_1 \omega_3+\omega_2 \omega_3 - a_{12} a_{21} - a_{23} a_{32}- a_{13} a_{31} = \lambda_1 \lambda_2+\lambda_1 \lambda_3+\lambda_2 \lambda_3. \]
Since all elements $a_{ij}$ are nonnegative, we have
\[ \omega_1 \omega_2+\omega_1 \omega_3+\omega_2 \omega_3  \ge \lambda_1 \lambda_2+\lambda_1 \lambda_3+\lambda_2 \lambda_3, \]
which is condition $(iv).$

Now, consider the expression
$(\lambda_1-\omega_1)(\lambda_2-\omega_1)
(\lambda_3-\omega_1).$
This is the value of $\Ch_1(\omega_1)$ and therefore, it must be the value of $\Ch_2(\omega_1)$. Hence,
\begin{multline*}
   (\lambda_1-\omega_1)(\lambda_2-\omega_1)
(\lambda_3-\omega_1) \\
=a_{12} a_{23} a_{31} + a_{13} a_{21} a_{32}
+a_{13} a_{31}(\omega_1-\omega_2)
+a_{23} a_{32}(\omega_1-\omega_3),
  \end{multline*}
which is nonnegative  because $\omega_1$ is larger than or equal to $\omega_2$ and $\omega_3$ by assumption. From $\lambda_1 \ge \lambda_2 \ge\lambda_3$, we have $\omega_1 \le \lambda_3$ or $\lambda_1 \geq \omega_1 \ge \lambda_2$. If $\omega_1 \le \lambda_3$, then for $\omega_1+ \omega_2+ \omega_3 = \lambda_1 + \lambda_2 +\lambda_3$ to be true, it must be that $\omega_1= \omega_2= \omega_3 = \lambda_1 = \lambda_2 =\lambda_3$. Hence, this case is resolved in the case that $\lambda_1 \geq \omega_1 \ge \lambda_2$, which is condition $(ii).$ Therefore, a nonnegative matrix with eigenvalues $\Lambda$ and  diagonal entries $\Omega$
satisfies  conditions $(i)$--$(iv)$.

Now, suppose that lists $\Lambda$ and  $\Omega$ satisfy conditions $(i)$--$(iv)$.
Let $B$ be the following matrix
\[ B=\left(\begin{array}{cccccccc}
    \omega_1 & 0 & (\lambda_1-\omega_1)(\lambda_2-\omega_1)
         (\lambda_3-\omega_1)\\
     1 & \omega_2 & \omega_1 \omega_2+\omega_1 \omega_3+\omega_2 \omega_3 - \lambda_1 \lambda_2-\lambda_1 \lambda_3 -\lambda_2 \lambda_3   \\
     0 & 1 & \omega_3
            \end{array}\right).
\]
Then, $B$ is nonnegative by conditions $(i),\, (ii)$, and $(iv)$. The characteristic polynomial $\Ch_3(\lambda)$ of $B$ is factored as
\[
 \Ch_3(\lambda)=\det(B-\lambda I)
 =(\lambda_1-\lambda)(\lambda_2-\lambda)
(\lambda_3-\lambda),
 \]
where we use condition $(iii)$ for factorization. Therefore, the eigenvalues of $B$ are $\Lambda$, and this concludes statement \ref{enum:rg1}.

\item Assume that statement \ref{enum:rg1} holds.
    We first show that if number $\omega_1$ is within range (\ref{eq:rgw1}), then taking numbers $\omega_2$ and $\omega_3$ with expression (\ref{eq:rgw2}), list $\Omega$ satisfies conditions $(i)$--$(iv)$. Second, we show that if number $\omega_1$ is out of range (\ref{eq:rgw1}), $\Omega$ does not satisfy them.

    Suppose that $\Omega$ is within ranges (\ref{eq:rgw1}) and (\ref{eq:rgw2}).
    Then, we have  $\omega_1-\omega_2=\frac{1}{2}(3 \omega_1- \lambda_1 - \lambda_2-\lambda_3),$ which is nonnegative by (\ref{eq:rgw1}). Hence, we obtain $\omega_1 \ge \omega_2 \ge \omega_3$. From an $\omega_1$ within range (\ref{eq:rgw1}), we may compute that
    \[ \max \{ 0, \frac{1}{2}(\lambda_2+\lambda_3)\} \le \omega_3 \le \min \{\frac{1}{3}(\lambda_1+ \lambda_2+\lambda_3), \frac{1}{2}(\lambda_1+\lambda_3) \}. \]
    Hence, $\omega_3 \ge 0 $ and therefore, condition $(i)$ holds.
    For a number $\omega_1$ within range (\ref{eq:rgw1}), we have $\lambda_2 \le \omega_1 \le \lambda_1$, which is condition $(ii)$. Furthermore, a direct summation of $\omega_1, \omega_2$, and $\omega_3$ produces condition $(iii)$.
    From numbers $\omega_1, \omega_2$, and $\omega_3$ within ranges (\ref{eq:rgw1}) and (\ref{eq:rgw2}), we obtain
 \begin{multline*}
    \max \{-\lambda_1 \lambda_2-\lambda_1 \lambda_3
    -\lambda_2 \lambda_3, \frac{1}{4}(\lambda_2-\lambda_3)^2 \} \\
    \le \omega_1 \omega_2+\omega_1 \omega_3+\omega_2 \omega_3 -\lambda_1 \lambda_2-\lambda_1 \lambda_3 -\lambda_2 \lambda_3 .
 \end{multline*}
    Since $\max \{-\lambda_1 \lambda_2-\lambda_1 \lambda_3
    -\lambda_2 \lambda_3, \frac{1}{4}(\lambda_2-\lambda_3)^2 \} $ is nonnegative, condition $(iv)$ is satisfied.  Therefore, if list $\Omega$ is within ranges (\ref{eq:rgw1}) and (\ref{eq:rgw2}), then $\Omega$ satisfies conditions $(i)$--$(iv)$.

    Now, consider the case when $\omega_1$ is out of range (\ref{eq:rgw1}). If $\omega_1 < \frac{1}{3}(\lambda_1 +\lambda_2+\lambda_3)$, then
    \[ \lambda_1 +\lambda_2+\lambda_3 > 3 \omega_1 \ge \omega_1+ \omega_2+\omega_3,
    \]
    which violates condition $(iii)$.
    If $\omega_1 < \lambda_2$, then it violates condition $(ii)$.
    If $\omega_1 > \lambda_1 +\lambda_2+\lambda_3$, then since $\omega_2$ and $\omega_3$ are nonnegative, $\omega_1+\omega_2 +\omega_3 > \lambda_1 + \lambda_2+ \lambda_3$, which violates condition $(iii)$. If $\omega_1 > \lambda_1$, then it violates condition $(ii)$.
    Therefore, if $\omega_1$ is out of range (\ref{eq:rgw1}), then $\Omega$ does not satisfy one of conditions $(i)$--$(iv)$. Therefore, we conclude statement \ref{enum:rg3}.

\item From statement \ref{enum:rg1},  statement \ref{enum:rg4} is equivalent to the following: for a given list $\Lambda$, there is a list $\Omega$ with conditions
 $(i)$--$(iv)$ of statement \ref{enum:rg1} if and only if $\Lambda$ satisfies conditions $\lambda_1 +\lambda_3 \ge 0$ and $\lambda_1 + \lambda_2 +\lambda_3 \ge 0.$ Hence, we give a proof for this statement.

    Suppose that there is a list $\Omega$ that satisfies conditions $(i)$--$(iv)$ of statement \ref{enum:rg1}. Since the elements of $\Omega$ are nonnegative, by condition $(iii)$ of statement \ref{enum:rg1}, we have $\lambda_1 + \lambda_2 +\lambda_3 \ge 0 $. Again, from condition $(iii)$ of statement \ref{enum:rg1}, we have $\lambda_1 + \lambda_3 = (\omega_1-\lambda_2)+\omega_2+\omega_3$. Then, by conditions $(i)$ and $(ii)$ of statement \ref{enum:rg1}, we obtain $\lambda_1 + \lambda_3 \ge 0$.

    Conversely, suppose that $\lambda_1 +\lambda_3 \ge 0$ and $\lambda_1 + \lambda_2 +\lambda_3 \ge 0. $ Then, we may take values $\omega_1, \omega_2$, and $\omega_3$ to be in ranges (\ref{eq:rgw1}) and (\ref{eq:rgw2}). For example, let $\omega_1=\min\{\lambda_1 +\lambda_2+\lambda_3, \lambda_1\}$ and $\omega_2=\omega_3=\frac{1}{2}(\lambda_1 +\lambda_2+\lambda_3 -\omega_1)$. Then, by statement \ref{enum:rg3}, $\Omega$ satisfies conditions $(i)$--$(iv)$ of  statement \ref{enum:rg1}. Hence, we conclude statement \ref{enum:rg4}.

\end{enumerate}
\end{proof}

%
%
Consider the case of complex eigenvalues. When we say a list of complex numbers, we mean that the list has at least one non-real complex number.    When a list $\Lambda=\{\lambda_1, \lambda_2, \lambda_3\}$ of complex numbers is realizable as the eigenvalues of a nonnegative matrix, from the Perron-Frobenius theorem, we have that $ \lambda_1 \ge |\lambda_2| \ge |\lambda_3|$. From now on, when we treat a complex number list, we make the hypothesis that $ \lambda_1 \ge |\lambda_2| \ge |\lambda_3|$.

\begin{theorem}[complex, general case]\label{the:cg}
  Let $\Lambda=\{\lambda_1, \lambda_2, \lambda_3\}$ be a list of complex numbers with $\lambda_1 \ge |\lambda_2| \ge |\lambda_3|$ and let
$\Omega=\{\omega_1, \omega_2, \omega_3\}$ be a list of real numbers with $\omega_1 \ge\omega_2\ge \omega_3$.

\begin{enumerate}
\item\label{enum:cg1} There is a nonnegative matrix with eigenvalues $\Lambda$ and diagonal entries
$\Omega$ if and only if lists $\Lambda$ and $\Omega$ satisfy
 \begin{description}
    \item[] (i) $\omega_i \ge 0$ for $i= 1, 2, \hbox{\ and\ } 3$,
    \item[] (ii) $\lambda_1 \geq \omega_1$,
    \item[] (iii) $\omega_1+~\omega_2 +\omega_3 = \lambda_1 + \lambda_2+ \lambda_3$,
    \item[] (iv) $\omega_1 \omega_2+\omega_1 \omega_3+\omega_2 \omega_3 \ge \lambda_1 \lambda_2+\lambda_1 \lambda_3 +\lambda_2 \lambda_3$.
 \end{description}

\item\label{enum:cg3} If $\Lambda$ and $\Omega$ satisfy the conditions of statement \ref{enum:cg1}, the exact range for $\omega_1$ is
 \begin{equation}\label{eq:cgw1}
        \frac{1}{3}(\lambda_1 +\lambda_2+\lambda_3) \le \omega_1 \le \min\{\lambda_1 +\lambda_2+\lambda_3, U_1 \},
 \end{equation}
 where
 \begin{equation}\label{eq:U1}
 U_1= \frac{1}{3}(\lambda_1 +\lambda_2+\lambda_3) + \frac{2}{3} \sqrt{\lambda_1^2 +\lambda_2^2 +\lambda_3^2 - \lambda_1 \lambda_2 -\lambda_1 \lambda_3 - \lambda_2 \lambda_3}.
 \end{equation}
 For the values for $\omega_2$ and $\omega_3$, we may take
        \begin{equation}\label{eq:cgw2}
        \omega_2=\omega_3=\frac{1}{2}(\lambda_1 +\lambda_2+\lambda_3 -\omega_1).
        \end{equation}

\item \cite{LoewyLondon1978}\label{enum:cg4} The list $\Lambda$ is realizable as the eigenvalues of a nonnegative matrix
if and only if the following hold:
    \begin{description}
        \item[] (i) $\lambda_1 \ge 0$ and $\lambda_2 =  \overline{\lambda_3}$,
        \item[] (ii) $\lambda_1 +\lambda_2+\lambda_3 \ge 0$,
        \item[] (iii) $\lambda_1^2 +\lambda_2^2 +\lambda_3^2 \ge  \lambda_1 \lambda_2 +\lambda_1 \lambda_3 + \lambda_2 \lambda_3$.
    \end{description}

\end{enumerate}
\end{theorem}

\label{enum:cg2} For $\Lambda$ and $\Omega$ satisfying the conditions of statement \ref{enum:cg1}, we find the following nonnegative matrix
  \[  \left(\begin{array}{cccccccc}
    \omega_1 & 0 & (\lambda_1-\omega_1)(\lambda_2-\omega_1)
         (\lambda_3-\omega_1)\\
     1 & \omega_2 & \omega_1 \omega_2+\omega_1 \omega_3+\omega_2 \omega_3 - \lambda_1 \lambda_2-\lambda_1 \lambda_3 -\lambda_2 \lambda_3   \\
     0 & 1 & \omega_3
            \end{array}\right)
   \]
 with eigenvalues $\Lambda$ and diagonal entries $\Omega$.

\begin{proof}
\begin{enumerate}
\item
    The proof of statement \ref{enum:cg1}
is almost identical to that of statement \ref{enum:rg1} of Theorem \ref{the:rg},
even though the list $\Lambda$ contains complex numbers.

\item Assuming that the statement \ref{enum:cg1} holds,
    we first show that if the number $\omega_1$ is within range (\ref{eq:cgw1}), then taking the numbers $\omega_2$ and $\omega_3$ with expression (\ref{eq:cgw2}), the list $\Omega$ satisfies conditions $(i)$--$(iv)$. Second, we show that when $\omega_1$ is out of range (\ref{eq:cgw1}), $\Omega$ does not satisfy them.

    Suppose that $\Omega$ is within range (\ref{eq:cgw1}) and (\ref{eq:cgw2}). From expression (\ref{eq:cgw2}),
    $\omega_1-\omega_2=\frac{1}{2}(3 \omega_1- \lambda_1 - \lambda_2-\lambda_3)$,
which is nonnegative by (\ref{eq:cgw1}). Hence, we obtain $\omega_1 \ge \omega_2 \ge \omega_3$.
For $\omega_1$ within range (\ref{eq:cgw1}), we may compute that
    \begin{multline*}
    \max \{ 0, \frac{1}{3}\left(\lambda_1 +\lambda_2+\lambda_3 -\sqrt{\lambda_1^2 +\lambda_2^2 +\lambda_3^2 - \lambda_1 \lambda_2 -\lambda_1 \lambda_3 - \lambda_2 \lambda_3}\right)\} \\
    \le \omega_3 \le  \frac{1}{3}(\lambda_1+ \lambda_2+\lambda_3).
 \end{multline*}
    Hence, $\omega_3 \ge 0$, and therefore, condition $(i)$ holds.

    When $\omega_1$ is within range (\ref{eq:cgw1}), we have
    $ \omega_1 \le  U_1.$
    Furthermore,  $ U_1 \le \lambda_1$ because
    \begin{equation*} \lambda_1 - U_1 \\ =
    \frac{2}{3} \lambda_1 -  \frac{1}{3} \lambda_2 - \frac{1}{3} \lambda_3 -  \frac{2}{3}\sqrt{\lambda_1^2 +\lambda_2^2 +\lambda_3^2 - \lambda_1 \lambda_2 -\lambda_1 \lambda_3 - \lambda_2 \lambda_3},
    \end{equation*} and
    \begin{equation*} \left(\frac{2}{3} \lambda_1 -  \frac{1}{3} \lambda_2 - \frac{1}{3} \lambda_3\right)^2 -  \left(\frac{2}{3}\sqrt{\lambda_1^2 +\lambda_2^2 +\lambda_3^2 - \lambda_1 \lambda_2 -\lambda_1 \lambda_3 - \lambda_2 \lambda_3}\right)^2 \\ =       - (\lambda_2 - \lambda_3)^2,
    \end{equation*}
    which is nonnegative from the equality $\lambda_2 =  \overline{\lambda_3}$. 
     Hence, we have $\omega_1 \le \lambda_1$, and therefore, condition $(ii)$ holds.
    The summation of $\omega_1, \omega_2$, and $\omega_3$ gives us condition $(iii)$.
    From the numbers $\omega_1, \omega_2$, and $\omega_3$, we obtain
 \begin{multline*}
    \max \{-\lambda_1 \lambda_2-\lambda_1 \lambda_3
    -\lambda_2 \lambda_3, 0 \} \\
    \le \omega_1 \omega_2+\omega_1 \omega_3+\omega_2 \omega_3 -\lambda_1 \lambda_2-\lambda_1 \lambda_3 -\lambda_2 \lambda_3 \le \\
    \frac{1}{3}(\lambda_1^2 +\lambda_2^2 +\lambda_3^2 -\lambda_1 \lambda_2-\lambda_1 \lambda_3
    -\lambda_2 \lambda_3).
 \end{multline*}
    Thus, condition $(iv)$ is satisfied. Therefore, if the list $\Omega$ is within ranges (\ref{eq:cgw1}) and (\ref{eq:cgw2}),
then $\Omega$ satisfies conditions $(i)$--$(iv)$.

    Now, we consider the case when $\omega_1$ is out of range (\ref{eq:cgw1}).
If $\omega_1 < \frac{1}{3}(\lambda_1 +\lambda_2+\lambda_3)$, then
    \[ \lambda_1 +\lambda_2+\lambda_3 > 3 \omega_1 \ge \omega_1+ \omega_2+\omega_3,
    \]
    which violates condition $(iii)$.
    If $\omega_1 > \lambda_1 +\lambda_2+\lambda_3$, then, because $\omega_2$ and $\omega_3$ are nonnegative,  condition $(iii)$ is violated.

    Now, let us consider the case when $\omega_1 >  U_1$.
    Using $\omega_3 = \lambda_1 +\lambda_2+\lambda_3-\omega_1-\omega_2$, we have
    \begin{multline*}
    \omega_1 \omega_2+\omega_1 \omega_3+\omega_2 \omega_3 - \lambda_1 \lambda_2-\lambda_1 \lambda_3 -\lambda_2 \lambda_3 \\
    =
    -\omega_2^2 +(\lambda_1 + \lambda_2+ \lambda_3 -\omega_1)\omega_2 \\
    +(\lambda_1 + \lambda_2+ \lambda_3)\omega_1
    -(\lambda_1 \lambda_2 +\lambda_1 \lambda_3 + \lambda_2 \lambda_3+\omega_1^2).
    \end{multline*}
    From statement $(iii)$, we have $ \omega_1+2\omega_2 \ge \lambda_1 + \lambda_2+ \lambda_3$, hence $\omega_2 \ge \frac{1}{2}
    (\lambda_1 + \lambda_2+ \lambda_3-\omega_1)$.
    Applying this inequality to the above expression, we have
    \begin{multline*}
     \omega_1 \omega_2+\omega_1 \omega_3+\omega_2 \omega_3 - \lambda_1 \lambda_2-\lambda_1 \lambda_3 -\lambda_2 \lambda_3 \\
     \le -\frac{3}{4} \left( \omega_1^2 -\frac{2}{3}(\lambda_1 + \lambda_2+ \lambda_3)\omega_1 -\frac{1}{3}(\lambda_1^2 +\lambda_2^2 +\lambda_3^2 -2 \lambda_1 \lambda_2 -2 \lambda_1 \lambda_3 -2 \lambda_2 \lambda_3)\right)\\
     =- \frac{3}{4} \left(\omega_1 - \frac{1}{3}
     (\lambda_1 +\lambda_2+\lambda_3) + \frac{2}{3} \sqrt{\lambda_1^2 +\lambda_2^2 +\lambda_3^2 - \lambda_1 \lambda_2 -\lambda_1 \lambda_3 - \lambda_2 \lambda_3}\right) \\
     \left(\omega_1 - \frac{1}{3}
     (\lambda_1 +\lambda_2+\lambda_3) -\frac{2}{3} \sqrt{\lambda_1^2 +\lambda_2^2 +\lambda_3^2 - \lambda_1 \lambda_2 -\lambda_1 \lambda_3 - \lambda_2 \lambda_3}\right).
    \end{multline*}
    Hence, for $\omega_1 > U_1$, we have
    $\omega_1 \omega_2+\omega_1 \omega_3+\omega_2 \omega_3 - \lambda_1 \lambda_2-\lambda_1 \lambda_3 -\lambda_2 \lambda_3 <0$, which violates condition $(iv)$ of statement \ref{enum:cg1}. Therefore, if $\omega_1$ is out of range (\ref{eq:cgw1}), then $\Omega$ does not satisfy conditions
 $(i)$--$(iv)$. Hence, we conclude statement \ref{enum:cg3}.

\item From statement \ref{enum:cg1}, statement \ref{enum:cg4} is equivalent to
the following: for a given list $\Lambda$, there is a list $\Omega$ with conditions
$(i)$--$(iv)$ of statement \ref{enum:cg1} if and only if $\Lambda$ satisfies conditions $(i)$--$(iii)$ of
statement \ref{enum:cg4}. Hence, we give a proof for this statement.

    Suppose that $\Omega$ satisfies conditions $(i)$--$(iv)$ of statement \ref{enum:cg1}.
By condition $(i)$ of statement \ref{enum:cg1}, we have $\lambda_1 \ge 0$.
Since $\lambda_1 +\lambda_2+\lambda_3$ is real, $\Im(\lambda_2)=-\Im(\lambda_3)$.
Since $0=\Im(\lambda_1 \lambda_2 +\lambda_1 \lambda_3 + \lambda_2 \lambda_3)=(\Re(\lambda_2)-\Re(\lambda_3)) \Im(\lambda_2)$
and $\Lambda$ contains complex conjugates, we conclude that $\Re(\lambda_2)=\Re(\lambda_3)$.
Therefore, $\lambda_2 =  \overline{\lambda_3}$, which is condition $(i)$. Since the elements of $\Omega$ are nonnegative,
we have $\lambda_1 + \lambda_2 +\lambda_3 \ge 0 $, which is condition $(ii)$.

    Again, using condition $(iii)$ of statement \ref{enum:cg1}, we have
    \[ \begin{aligned}
    \lambda_1^2 +\lambda_2^2 +\lambda_3^2 & - \lambda_1 \lambda_2 -\lambda_1 \lambda_3 - \lambda_2 \lambda_3 \\
    & =(\lambda_1 +\lambda_2 +\lambda_3)^2 -
    3(\lambda_1 \lambda_2 +\lambda_1 \lambda_3 + \lambda_2 \lambda_3) \\
    & = (\omega_1 +\omega_2+ \omega_3)^2
    -    3(\lambda_1 \lambda_2 +\lambda_1 \lambda_3 + \lambda_2 \lambda_3) \\
    & =\frac{1}{2}\left((\omega_1 -\omega_2)^2 +(\omega_1-\omega_3)^2 +(\omega_2 -\omega_3)^2\right) \\
    & \qquad +3(\omega_1 \omega_2+\omega_1 \omega_3+\omega_2 \omega_3 - \lambda_1 \lambda_2-\lambda_1 \lambda_3 -\lambda_2 \lambda_3),
    \end{aligned} \]
    which is nonnegative by condition $(iv)$ of statement \ref{enum:cg1}.
    Therefore, $\Lambda$ satisfies conditions $(i)$--$(iii)$ of statement \ref{enum:cg4}.

    Conversely, suppose that $\Lambda$ satisfies conditions $(i)$--$(iii)$ of statement \ref{enum:cg4}.
Then, we may take values $\omega_1, \omega_2$, and $\omega_3$ from ranges (\ref{eq:cgw1}) and (\ref{eq:cgw2}).
For example, take $\omega_1=\omega_2=\omega_3=\frac{1}{3}(\lambda_1 +\lambda_2+\lambda_3)$.
Then, by statement \ref{enum:cg3}, $\Omega$ satisfies conditions
$(i)$--$(iv)$ of statement \ref{enum:cg1}. Hence, we conclude statement \ref{enum:cg4}.

\end{enumerate}
\end{proof}

From statement \ref{enum:cg1} of Theorem \ref{the:cg},
we may write $\lambda_1=a, \lambda_2 =b+ci$ and $\lambda_3=b-ci$
where $a, b$, and $c$ are real numbers. We restate statements \ref{enum:cg3}
and \ref{enum:cg4} of Theorem \ref{the:cg} in terms of $a, b$, and $c$.
Note that $\Lambda$ is now $\{ a, b+c i, b-c i \}$.

\begin{corollary}
\begin{enumerate}
\item The list $\Lambda$ is realizable as the eigenvalues of a nonnegative matrix if and only
if the following hold:
    \begin{description}
        \item[] (i) $a \ge 0$,
        \item[] (ii) $-\frac{a}{2}  \le b \le a$,
        \item[] (iii) $(a-b)^2 \ge 3c^2 $.
    \end{description}

\item The exact range for $\omega_1$ is
    \[
    \frac{1}{3}(a+2b) \le \omega_1 \le \min\{a+2b, \frac{1}{3}(a+2b)+\frac{2}{3}\sqrt{(a-b)^2-3c^2}\}.
    \]
\end{enumerate}
\end{corollary}

%
%
%

\section{Symmetric $3 \times 3$ nonnegative matrices} \label{sec:Symmetric}

We discuss a symmetric $3 \times 3$ nonnegative matrix case with prescribed eigenvalues and diagonal entries.  The following theorem was proved by
Fieldler \cite{Fiedler1974} in 1974, except for statement \ref{enum:sy3}, but we include here the theorem and proof to compare the other results and for completeness.

\begin{theorem}[symmetric case]
  \label{the:sy}
  Let $\Lambda=\{\lambda_1, \lambda_2, \lambda_3\}$ and $\Omega=\{\omega_1, \omega_2, \omega_3\}$
be lists of real numbers with $\lambda_1 \ge \lambda_2 \ge\lambda_3$ and $\omega_1 \ge\omega_2\ge \omega_3$.

\begin{enumerate}
\item\label{enum:sy1}\cite{Fiedler1974} There is a symmetric nonnegative matrix with eigenvalues $\Lambda$ and diagonal entries $\Omega$ if and only if
$\Lambda$ and $\Omega$ satisfy
 \begin{description}
    \item [] (i) $\omega_i \ge 0$ for $i= $ 1, 2, and 3,
    \item [] (ii) $\lambda_1 \geq \omega_1 \ge \lambda_2$,
    \item [] (iii) $\lambda_1 + \lambda_2 \geq \omega_1+\omega_2 $,
    \item [] (iv) $\lambda_1 + \lambda_2+ \lambda_3 =\omega_1+\omega_2 +\omega_3$.
 \end{description}

 \item\label{enum:sy3} If $\Lambda$ and $\Omega$ satisfy the conditions of statement \ref{enum:sy1},
then the exact range for $\omega_1$ is       \begin{equation}\label{eq:syw1}
        \max\{\frac{1}{3}(\lambda_1 +\lambda_2+\lambda_3), \lambda_2\} \le \omega_1 \le \min\{\lambda_1 +\lambda_2+\lambda_3, \lambda_1\}.
        \end{equation}
For the values for $\omega_2$ and $\omega_3$, we may take
        \begin{equation}\label{eq:syw2}
        \omega_2=\omega_3=\frac{1}{2}(\lambda_1 +\lambda_2+\lambda_3 -\omega_1).
        \end{equation}

\item\label{enum:sy4}\cite{Fiedler1974} The list $\Lambda$ is realizable as the eigenvalues of a symmetric nonnegative matrix if and only if the following hold:
    \begin{description}
        \item[] (i) $\lambda_1 +\lambda_3 \ge 0$,
        \item[] (ii) $\lambda_1 + \lambda_2 +\lambda_3 \ge 0. $
    \end{description}

 \end{enumerate}
\end{theorem}

\label{enum:sy2} For $\Lambda$ and $\Omega$ satisfying the conditions of statement \ref{enum:sy1}, we find the following symmetric nonnegative matrix
    \[ \left(\begin{array}{cccccccc}
    \omega_1 & \sqrt{\frac{\beta \gamma}{\alpha+\beta}}
    & \sqrt{\frac{\alpha \gamma}{\alpha+\beta}}\\
    \sqrt{\frac{\beta \gamma}{\alpha+\beta}}
    & \omega_2 &\sqrt{ \alpha \beta}\\
    \sqrt{\frac{\alpha \gamma}{\alpha+\beta}} & \sqrt{ \alpha \beta} & \omega_3
            \end{array}\right),
    \]
    where $\alpha = \lambda_1 + \lambda_2 - \omega_1 - \omega_2,\ \beta = \lambda_1 + \lambda_2 - \omega_1 - \omega_3$, and $\gamma = (\lambda_1 -\omega_1) (\omega_1 - \lambda_2)$, with eigenvalues $\Lambda$ and diagonal entries $\Omega$.

\begin{proof}
\begin{enumerate}
\item The proof is similar to that of statement \ref{enum:rg1} of Theorem \ref{the:rg} except for condition $(iii)$. Therefore, we give the proof related to condition $(iii)$. Assume that the list $\Lambda$ is realizable as the eigenvalues of a symmetric nonnegative matrix $A$ with diagonal entries $\Omega$. Then, there is an orthogonal matrix $U=\{ u_{ij} \}$ such that $A=U^{\text{tr}} [\Lambda] U$ where $[\Lambda]$ is the diagonal matrix with diagonal entries $\Lambda$. Then, in particular, we have
    \[ \omega_3 = u_{13}^2 \lambda_1 + u_{23}^2 \lambda_2 + u_{33}^2 \lambda_3. \]
    Since $\lambda_1 \ge \lambda_2 \ge\lambda_3$, we have \[ \omega_3 \ge (u_{13}^2  + u_{23}^2  + u_{33}^2) \lambda_3. \]
    Since $U$ is orthogonal, $u_{13}^2  + u_{23}^2  + u_{33}^2 =1$, and hence $\omega_3 \ge \lambda_3$. From condition $(iv)$, we obtain
    \[ \lambda_1 + \lambda_2 \ge \omega_1+ \omega_2,  \]
    which is condition $(iii)$.

    Now, suppose that the lists $\Lambda$ and  $\Omega$ satisfy conditions $(i)$--$(iv)$.
    Let $B$ be the following matrix
    \[ B=\left(\begin{array}{cccccccc}
    \omega_1 & \sqrt{\frac{\beta\, \gamma}{\alpha+\beta}} & \sqrt{\frac{\alpha \, \gamma}{\alpha+\beta}}\\
    \sqrt{\frac{\beta \, \gamma}{\alpha+\beta}}
    & \omega_2 &\sqrt{ \alpha \, \beta}\\
    \sqrt{\frac{\alpha \, \gamma}{\alpha+\beta}} & \sqrt{ \alpha \, \beta} & \omega_3
            \end{array}\right),
    \]
    where $\alpha = \lambda_1 + \lambda_2 - \omega_1 - \omega_2,\ \beta = \lambda_1 + \lambda_2 - \omega_1 - \omega_3$, and $\gamma = (\lambda_1 -\omega_1) (\omega_1 - \lambda_2)$. Then, $B$ is a symmetric nonnegative matrix by conditions $(i)$--$(iii) $. The characteristic polynomial $\Ch(\lambda)$ of $B$ is factored as
\[
 \Ch(\lambda)=\det(B-\lambda I)
 =(\lambda_1-\lambda)(\lambda_2-\lambda)
(\lambda_3-\lambda),
 \]
where we use condition $(iv)$. Therefore, the eigenvalues of $B$ are $\Lambda$, hence this concludes statement \ref{enum:sy1}.

\item[2, 3.] Suppose that $\Omega$ is within ranges (\ref{eq:syw1}) and (\ref{eq:syw2}). Then, a direct computation shows that
\begin{equation*}
    \max \{\frac{1}{2}( \lambda_2- \lambda_3),
    -\lambda_3\}
    \le \lambda_1 + \lambda_2 -\omega_1 -\omega_2 .
 \end{equation*}
 Therefore, $\lambda_1 + \lambda_2 -\omega_1 -\omega_2$ is nonnegative, which is condition $(iii)$. For the other part of the proof, see the proofs of statements \ref{enum:rg3} and \ref{enum:rg4} of Theorem \ref{the:rg}.

\end{enumerate}
\end{proof}

There appear two conditions: $ \omega_1 \omega_2+\omega_1 \omega_3+\omega_2 \omega_3 \ge \lambda_1 \lambda_2+\lambda_1 \lambda_3 +\lambda_2 \lambda_3$ in Theorem \ref{the:rg} and $\lambda_1 + \lambda_2 \ge \omega_1+\omega_2$ in Theorem \ref{the:sy}. It is interesting to compare conditions.

\begin{proposition}\label{pro:relofcond}
Let $\Lambda=\{\lambda_1, \lambda_2, \lambda_3\}$ and $\Omega=\{\omega_1, \omega_2, \omega_3\}$
be lists of real numbers with $\lambda_1 \ge \lambda_2 \ge\lambda_3$ and $\omega_1 \ge\omega_2\ge \omega_3$.
Suppose that  $\lambda_1 \geq \omega_1 \ge \lambda_2$ and $\omega_1+\omega_2 +\omega_3 = \lambda_1 + \lambda_2+ \lambda_3$.
Then, $\lambda_1 + \lambda_2 \ge \omega_1+\omega_2$
implies \[ \omega_1 \omega_2+\omega_1 \omega_3+\omega_2 \omega_3 \ge \lambda_1 \lambda_2+\lambda_1 \lambda_3 +\lambda_2 \lambda_3.\]
\end{proposition}

\begin{proof}
We may write that
\begin{multline*}
    \omega_1 \omega_2+\omega_1 \omega_3+\omega_2 \omega_3 - \lambda_1 \lambda_2-\lambda_1 \lambda_3 -\lambda_2 \lambda_3\\
    = ( \lambda_1 -\omega_1) (\omega_1 -\lambda_2)+
      ( \omega_2 -\lambda_3) ( \omega_3 -\lambda_3) \\
    +(\omega_1+\lambda_3)(\omega_1+\omega_2 +\omega_3 - \lambda_1 - \lambda_2- \lambda_3) .
 \end{multline*}
 Since $\lambda_1 + \lambda_2 \ge \omega_1+\omega_2$ implies $\omega_3 \ge \lambda_3$, each term of the right-hand side of the equality is nonnegative, hence we conclude the proposition.
\end{proof}

\begin{remark}
  The converse of Proposition \ref{pro:relofcond} is false. Let
  $\Lambda= \{ 1, \frac{1}{2},\frac{1}{4} \}$ and
  $\Omega= \{ \frac{4}{5}, \frac{3}{4},\frac{1}{5}\}$.
  Then, $\omega_1 \omega_2+\omega_1 \omega_3+\omega_2 \omega_3 - \lambda_1 \lambda_2-\lambda_1 \lambda_3 -\lambda_2 \lambda_3 = \frac{7}{100}$ is positive, but $\lambda_1 + \lambda_2 - \omega_1-\omega_2 = -\frac{1}{20}$ is negative.
\end{remark}

%
%
%

\section{Stochastic $3 \times 3$ nonnegative matrices}\label{sec:Stochastic}

We discuss a stochastic $3 \times 3$ nonnegative matrix case with prescribed real or complex eigenvalues and diagonal entries. Consider first the case of real eigenvalues.  The following theorem was proved by
Perfect \cite{Perfect55} in 1955, except for statement \ref{enum:rs3}, but we include here the theorem and proof to compare the other results and for completeness.

Throughout the paper, we consider a positive scalar multiple of a stochastic matrix or a doubly stochastic matrix, not just a usual stochastic matrix.  Therefore the sum of entries in every row is the largest eigenvalue, say $\lambda_1$, and for the doubly stochastic matrix, the sum of entries in every column is $\lambda_1$ too. In this way, we can trace the behavior of $\lambda_1$, otherwise it is concealed when $\lambda_1=1$. From now on, a \emph{stochastic} matrix or a \emph{doubly stochastic} matrix means a positive scalar multiple of a stochastic matrix or a doubly stochastic matrix, respectively.

\begin{theorem}[real, stochastic case]\label{the:rs}
  Let $\Lambda=\{\lambda_1, \lambda_2, \lambda_3\}$ and $\Omega=\{\omega_1, \omega_2, \omega_3\}$
be lists of real numbers with $\lambda_1 \ge \lambda_2 \ge\lambda_3$ and $\omega_1 \ge\omega_2\ge \omega_3$.

\begin{enumerate}
\item\label{enum:rs1}\cite{Perfect55} There is a stochastic nonnegative matrix with eigenvalues $\Lambda$ and diagonal entries $\Omega$ if and only if the lists $\Lambda$ and $\Omega$ satisfy
 \begin{description}
    \item [] (i) $\omega_i \ge 0$ for $i= $ 1, 2, and 3,
    \item [] (ii) $\lambda_1 \geq \omega_1 \ge \lambda_2$,
    \item [] (iii) $\omega_1+~\omega_2 +\omega_3 = \lambda_1 + \lambda_2+ \lambda_3$,
    \item[] (iv) $\omega_1 \omega_2+\omega_1 \omega_3+\omega_2 \omega_3 \ge \lambda_1 \lambda_2+\lambda_1 \lambda_3 +\lambda_2 \lambda_3.$
 \end{description}

\item\label{enum:rs3} If $\Lambda$ and $\Omega$ satisfy the conditions of statement \ref{enum:rs1},
then the exact range for $\omega_1$ is
     \begin{equation}\label{eq:rsw1}
        \max\{\frac{1}{3}(\lambda_1 +\lambda_2+\lambda_3), \lambda_2\} \le \omega_1 \le \min\{\lambda_1 +\lambda_2+\lambda_3, \lambda_1\}
     \end{equation}
 For the values for $\omega_2$ and $\omega_3$, we may take
        \begin{equation}\label{eq:rsw2}
        \omega_2=\omega_3=\frac{1}{2}(\lambda_1 +\lambda_2+\lambda_3 -\omega_1).
        \end{equation}

\item\label{enum:rs4}\cite{Perfect55} The list $\Lambda$ is realizable as the eigenvalues of a stochastic nonnegative matrix if and only if the following hold:
    \begin{description}
        \item[] (i) $\lambda_1 +\lambda_3 \ge 0$,
        \item[] (ii) $\lambda_1 + \lambda_2 +\lambda_3 \ge 0. $
    \end{description}
\end{enumerate}
\end{theorem}

\label{enum:rs2} For $\Lambda$ and $\Omega$ satisfying the conditions of statement \ref{enum:rs1}, we find the following stochastic nonnegative matrix

  \[
  \left(\begin{array}{cccccccc}
  \omega_1 & 0 & \lambda_1-\omega_1 \\
  \lambda_1-\omega_2-p & \omega_2 & p \\
  0& \lambda_1-\omega_3  &  \omega_3
  \end{array}\right),
  \]
where
  \begin{equation}\label{eq:p}
    p = \begin{cases}
        \frac{1}{\lambda_1-\omega_3} (\omega_1 \omega_2+\omega_1 \omega_3+\omega_2 \omega_3 - \lambda_1 \lambda_2-\lambda_1 \lambda_3-\lambda_2 \lambda_3), & \lambda_1 \ne \omega_3; \\
        0, &  \lambda_1 = \omega_3,
        \end{cases}
  \end{equation}
with eigenvalues $\Lambda$ and diagonal entries $\Omega$.
For a usual stochastic matrix, we may take $\lambda_1=1$.

\begin{proof}
  Suppose that the list $\Lambda$ is realizable as the eigenvalues of a stochastic nonnegative matrix with diagonal entries $\Omega$. Since such a stochastic nonnegative matrix is a general nonnegative matrix, by Theorem \ref{the:rg}, the lists $\Omega$ and $\Lambda$ satisfy conditions $(i)$--$(iv)$.

  Now, suppose that the lists $\Omega$ and $\Lambda$ satisfy conditions $(i)$--$(iv)$. Let $C$ be the following stochastic matrix:
\[
  C = \left(\begin{array}{cccccccc}
  \omega_1 & 0 & \lambda_1-\omega_1 \\
  \lambda_1-\omega_2-p & \omega_2 & p \\
  0& \lambda_1-\omega_3  &  \omega_3
  \end{array}\right),
\]
where $p$ is defined in (\ref{eq:p}). Each entry of the matrix is nonnegative. We now check the nonnegativity of $\lambda_1-\omega_2-p$.
Consider the numerator of $\lambda_1-\omega_2-p$:
\begin{equation*}
 (\lambda_1-\omega_2)(\lambda_1-\omega_3)
 -(\omega_1 \omega_2+\omega_1 \omega_3+\omega_2 \omega_3 - \lambda_1 \lambda_2 -\lambda_1 \lambda_3 -\lambda_2 \lambda_3).
\end{equation*}
Using condition $(iii)$ and substituting $\lambda_1$ with $\omega_1 + \omega_2 + \omega_3 - \lambda_2 -\lambda_3$, the expression becomes
\[ (\omega_1-\lambda_2)(\omega_1-\lambda_3).
\]
Hence, we obtain
\begin{equation}\label{E:lop}
 \lambda_1-\omega_2-p = (\omega_1-\lambda_2)(\omega_1-\lambda_3)
 (\lambda_1-\omega_3)^{-1},
\end{equation}
which is nonnegative. Therefore, $C$ is nonnegative. A direct computation shows that the eigenvalues of $C$ are $\Lambda$, hence this concludes statement \ref{enum:rs1}.
For the proof of statements 2 and 3, refer to the proof of Theorem \ref{the:rg}.

\end{proof}

%
%
%

Consider the case of complex eigenvalues. The following theorem was proved by
by Soto, Salas, and Manzaneda \cite{Soto2010} in 2010, except for statement \ref{enum:cs3}, but we include here the theorem and proof to compare the other results and for completeness.
\begin{theorem}[complex, stochastic case]\label{the:cs}
  Let $\Lambda=\{\lambda_1, \lambda_2, \lambda_3\}$ be a list of complex numbers with $ \lambda_1 \ge |\lambda_2| \ge |\lambda_3|$ and  let $\Omega=\{\omega_1, \omega_2, \omega_3\}$ be a list of real numbers with $\omega_1 \ge\omega_2\ge \omega_3$.

\begin{enumerate}
\item\label{enum:cs1}\cite{Soto2010} There is a stochastic nonnegative matrix with eigenvalues $\Lambda$ and diagonal entries
$\Omega$ if and only if the lists $\Lambda$ and $\Omega$ satisfy
 \begin{description}
    \item [] (i) $\omega_i \ge 0$ for $i= $ 1, 2, and 3,
    \item [] (ii) $\lambda_1 \geq \omega_1$,
    \item [] (iii) $\omega_1+~\omega_2 +\omega_3 = \lambda_1 + \lambda_2+ \lambda_3$,
    \item[] (iv) $\omega_1 \omega_2+\omega_1 \omega_3+\omega_2 \omega_3 \ge \lambda_1 \lambda_2+\lambda_1 \lambda_3 +\lambda_2 \lambda_3.$
 \end{description}

\item\label{enum:cs3} If $\Lambda$ and $\Omega$ satisfy the conditions of statement \ref{enum:cs1}, the exact range for $\omega_1$ is
 \begin{equation}\label{eq:csw1}
        \frac{1}{3}(\lambda_1 +\lambda_2+\lambda_3) \le \omega_1 \le \min\{\lambda_1 +\lambda_2+\lambda_3, U_1 \},
 \end{equation}
 where $U_1$ is as defined in (\ref{eq:U1}).
   For the values for $\omega_2$ and $\omega_3$, we may take
         \begin{equation}\label{eq:csw2}
        \omega_2=\omega_3=\frac{1}{2}(\lambda_1 +\lambda_2+\lambda_3 -\omega_1).
        \end{equation}

\item\label{enum:cs4}\cite{Soto2010} The list $\Lambda$ is realizable as the eigenvalues of a stochastic nonnegative matrix if and only if the following hold:
    \begin{description}
        \item[] (i) $\lambda_1 \ge 0$ and $\lambda_2 =  \overline{\lambda_3}$,
        \item[] (ii) $\lambda_1 +\lambda_2+\lambda_3 \ge 0$,
        \item[] (iii) $\lambda_1^2 +\lambda_2^2 +\lambda_3^2 \ge  \lambda_1 \lambda_2 +\lambda_1 \lambda_3 + \lambda_2 \lambda_3$.
    \end{description}

\end{enumerate}
\end{theorem}

\label{enum:cs2} For $\Lambda$ and $\Omega$ satisfying the conditions of statement \ref{enum:cs1}, we find the following stochastic nonnegative matrix
  \[
  \left(\begin{array}{cccccccc}
  \omega_1 & 0 & \lambda_1-\omega_1 \\
  \lambda_1-\omega_2-p & \omega_2 & p \\
  0& \lambda_1-\omega_3  &  \omega_3
  \end{array}\right),
  \]
  where $p$ is in (\ref{eq:p}), with eigenvalues $\Lambda$ and diagonal entries $\Omega$.
  For a usual stochastic matrix, we may take $\lambda_1=1$.

\begin{proof}
The proof of this theorem is similar to that of Theorem \ref{the:cg} or \ref{the:rs} except for the argument about the equation (\ref{E:lop}). We consider that
\begin{align*}
   \lambda_1-\omega_2-p &= (\omega_1-\lambda_2)(\omega_1-\lambda_3)(\lambda_1-\omega_3)^{-1} \\
   &= \left( \Bigl( \omega_1-\Re(\lambda_2) \Bigr)^2 +(\Im(\lambda_2))^2 \right) (\lambda_1-\omega_3)^{-1},
\end{align*}
which is nonnegative.
This completes the proof of this Theorem.
\end{proof}

From statement \ref{enum:cs1} of Theorem \ref{the:cs}, without loss of generality we may write $\lambda_1=a, \lambda_2 =b+ci$, and $\lambda_3=b-ci$  where $a, b$, and $c$ are real numbers. We restate statements \ref{enum:cs3} and \ref{enum:cs4} of Theorem \ref{the:cs} in terms of $a, b$ and $c$. Note that the list $\Lambda$ is now $\{ a, b+c i, b-c i \}$.
\begin{corollary}
\begin{enumerate}
\item The list $\Lambda$ is realizable as the eigenvalues of a stochastic nonnegative matrix if and only if the following hold:
    \begin{description}
        \item[] (i) $a \ge 0$,
        \item[] (ii) $-\frac{a}{2}  \le b \le a$,
        \item[] (iii) $(a-b)^2 \ge 3c^2 $.
    \end{description}

\item  The exact range for $\omega_1$ is
    \[
    \frac{1}{3}(a+2b) \le \omega_1 \le \min\{a+2b, \frac{1}{3}(a+2b)+\frac{2}{3}\sqrt{(a-b)^2-3c^2}\}.
    \]
\end{enumerate}
\end{corollary}

%
%
%

\section{Symmetric stochastic $3 \times 3$ nonnegative matrices }\label{sec:Sym-stocha}

We discuss a symmetric stochastic $3 \times 3$ nonnegative matrix case with prescribed eigenvalues and diagonal entries.
Statement \ref{enum:sds4} in the following Theorem \ref{the:sds} was proved by Perfect and Minsky \cite[proof of Theorem 14]{PerfectMirsky65} in 1965.

\begin{theorem}[symmetric stochastic case]
  \label{the:sds}
   Let $\Lambda=\{\lambda_1, \lambda_2, \lambda_3\}$ and $\Omega=\{\omega_1, \omega_2, \omega_3\}$
be lists of real numbers with $\lambda_1 \ge \lambda_2 \ge\lambda_3$ and $\omega_1 \ge\omega_2\ge \omega_3$.

 \begin{enumerate}
 \item \label{enum:sds1} There is a symmetric stochastic nonnegative matrix with eigenvalues $\Lambda$ and diagonal entries $\Omega$ if and only if the lists $\Lambda$ and $\Omega$ satisfy
 \begin{description}
     \item [] (i) $\omega_i \ge 0$ for $i= $ 1, 2, and 3,
    \item [] (ii) $\omega_1+\omega_2 +\omega_3 = \lambda_1 + \lambda_2+ \lambda_3$,
    \item [] (iii) $\omega_3 - \frac{\lambda_2}{2}- \frac{\lambda_3}{2} \ge 0$,
   \item [] (iv) $(\omega_3-\frac{\lambda_2}{2}-\frac{\lambda_3}{2})^2
    =(\lambda_1-\omega_1)(\lambda_1-\omega_2) -\frac{1}{3}(\lambda_1-\lambda_2)(\lambda_1-\lambda_3)$.
 \end{description}

\item \label{enum:sds3}
 If $\Lambda$ and $\Omega$ satisfy the conditions of statement \ref{enum:sds1}, then the exact range for $\omega_1$ is
  \begin{equation}\label{eq:sdsw1}
        \max\{ L_1, L_2, L_3 \} \le \omega_1 \le \frac{1}{3}(\lambda_1+2 \lambda_2),
  \end{equation}
where
  \begin{align}\label{eq:Li}
  L_1&=\frac{1}{6} (2 \lambda_1 +3\lambda_2 +\lambda_3), \notag \\
  L_2&=\frac{1}{2}(\lambda_1 +\lambda_2+\lambda_3)+
  \frac{1}{2\sqrt{3}}\sqrt{-(\lambda_1 +2\lambda_2)(\lambda_1 +2 \lambda_3)} ,\\
  L_3&=\frac{1}{4}
  (2\lambda_1 + \lambda_2 +\lambda_3)
  +\frac{1}{4\sqrt{3}}\sqrt{-(2\lambda_1+\lambda_2-3\lambda_3) (2\lambda_1-3\lambda_2 +\lambda_3)}.  \notag
  \end{align}
 For the values for $\omega_2$ and $\omega_3$, we have
    \begin{equation}\label{eq:sdsw2}
        \omega_2=\frac{1}{2}(\lambda_1 +\lambda_2 +\lambda_3 -\omega_1)
        +\frac{1}{2\sqrt{3}} \sqrt{-
        (3\omega_1 -\lambda_1 -2\lambda_2)
        (3\omega_1 -\lambda_1 -2\lambda_3)},
    \end{equation}
  and
     \begin{equation}\label{eq:sdsw3}
        \omega_3=\frac{1}{2}(\lambda_1 +\lambda_2 +\lambda_3 -\omega_1)
        -\frac{1}{2\sqrt{3}} \sqrt{-
        (3\omega_1 -\lambda_1 -2\lambda_2)
        (3\omega_1 -\lambda_1 -2\lambda_3)}.
     \end{equation}

\item \cite{PerfectMirsky65} \label{enum:sds4} The list $\Lambda$ is realizable as the eigenvalues of a symmetric stochastic nonnegative matrix if and only if the following holds:
    \[ 2\lambda_1 + \lambda_2 + 3\lambda_3 \ge 0.
    \]

\end{enumerate}
\end{theorem}

\label{enum:sds2} For $\Lambda$ and $\Omega$ satisfying the conditions of statement \ref{enum:sds1}, we find the following symmetric stochastic nonnegative matrix
  \[
  \left(\begin{array}{cccccccc}
  \omega_1 & \omega_3-\frac{\lambda_2}{2}-\frac{\lambda_3}{2} & \omega_2-\frac{\lambda_2}{2}-\frac{\lambda_3}{2} \\
  \omega_3-\frac{\lambda_2}{2}-\frac{\lambda_3}{2} & \omega_2 & \omega_1-\frac{\lambda_2}{2}-\frac{\lambda_3}{2} \\
  \omega_2-\frac{\lambda_2}{2}-\frac{\lambda_3}{2}& \omega_1-\frac{\lambda_2}{2}-\frac{\lambda_3}{2} &  \omega_3
  \end{array}\right),
  \]
with eigenvalues $\Lambda$ and diagonal entries $\Omega$.
For a usual stochastic matrix, we may take $\lambda_1=1$.

\begin{proof}
\begin{enumerate}
\item
 Suppose that the list $\Lambda$ is realizable as the eigenvalues of a symmetric stochastic nonnegative matrix $A$ with diagonal entries $\Omega$. Then, the elements $\omega_1$, $\omega_2$, and $\omega_3$ are nonnegative, and therefore, we have condition $(i)$.
 The characteristic polynomial $\Ch_1(\lambda)$ of $A$, using the eigenvalues $\Lambda$, is
 \[ \Ch_1(\lambda) = \lambda_1 \lambda_2 \lambda_3 -(\lambda_1 \lambda_2+\lambda_1 \lambda_3 +\lambda_2 \lambda_3)\lambda+ (\lambda_1 + \lambda_2 + \lambda_3)\lambda^2-\lambda^3. \]

  We may write the symmetric stochastic nonnegative matrix $A$ without loss of generality as
  \[ A=\left(\begin{array}{cccccccc}
    \omega_1 & s & \lambda_1 - \omega_1 -s \\
    s  & \omega_2 & \lambda_1 - \omega_2 -s  \\
    \lambda_1 - \omega_1 -s  & \lambda_1 - \omega_2 -s  & \omega_3
             \end{array}\right),
  \]
 where each entry is nonnegative and $\omega_3=2s +\omega_1+\omega_2-\lambda_1$. Then, we have another form of the characteristic polynomial $\Ch_2(\lambda)$ of $A$, which is
  \begin{multline*}
   \Ch_2(\lambda) = (-3s^2 +(2s -\omega_1 -\omega_2) \lambda_1 + 3 \omega_1 \omega_2) \lambda_1\\
   - (-3s^2 +(4 s +\omega_1 + \omega_2 -2 \lambda_1)\lambda_1 + 3 \omega_1 \omega_2 )\lambda \\
   + (\omega_1 + \omega_2  + \omega_3 )\lambda^2
   - \lambda^3.
  \end{multline*}
Then, the two characteristic polynomials $\Ch_1(\lambda)$ and $\Ch_2(\lambda)$ must be identical. Comparing the coefficients of $\lambda^2$, we have
\[\lambda_1 + \lambda_2 +\lambda_3  = \omega_1+ \omega_2+ \omega_3, \]
which is condition $(ii)$.
From $\omega_3=2s +\omega_1+\omega_2-\lambda_1$, we have $s = \omega_3-\frac{\lambda_2}{2}-\frac{\lambda_3}{2}$, which is nonnegative, hence condition $(iii)$ holds.
Comparing the coefficients of $\lambda$, we have
\[\lambda_1 \lambda_2+\lambda_1 \lambda_3+\lambda_2 \lambda_3 = -3s^2 +(4 s +\omega_1 + \omega_2 -2 \lambda_1)\lambda_1 + 3 \omega_1 \omega_2.  \]
Then, substituting $s^2$ by $(\omega_3-\frac{\lambda_2}{2}-\frac{\lambda_3}{2})^2$ and $s$ by $\lambda_1+ \frac{\lambda_2}{2}+\frac{\lambda_3}{2}-\omega_1-\omega_2$, we obtain
\[ (\omega_3-\frac{\lambda_2}{2}-\frac{\lambda_3}{2})^2
    =(\lambda_1-\omega_1)(\lambda_1-\omega_2) -\frac{1}{3}(\lambda_1-\lambda_2)(\lambda_1-\lambda_3),
\]
which is condition $(iv)$.
Therefore, all conditions $(i)$--$(iv)$ are satisfied.

Now, suppose that lists $\Lambda$ and  $\Omega$ satisfy conditions  $(i)$--$(iv)$.
Let $B$ be the following  matrix
\[ B=
  \left(\begin{array}{cccccccc}
  \omega_1 & \omega_3-\frac{\lambda_2}{2}-\frac{\lambda_3}{2} & \omega_2-\frac{\lambda_2}{2}-\frac{\lambda_3}{2} \\
  \omega_3-\frac{\lambda_2}{2}-\frac{\lambda_3}{2} & \omega_2 & \omega_1-\frac{\lambda_2}{2}-\frac{\lambda_3}{2} \\
  \omega_2-\frac{\lambda_2}{2}-\frac{\lambda_3}{2}& \omega_1-\frac{\lambda_2}{2}-\frac{\lambda_3}{2} &  \omega_3
  \end{array}\right).
\]
Because of condition $(iii)$, which is $\omega_3 - \frac{\lambda_2}{2}- \frac{\lambda_3}{2} \ge 0$, and $\omega_1 \ge\omega_2\ge \omega_3\ge 0$,  we have that $\omega_1 - \frac{\lambda_2}{2}- \frac{\lambda_3}{2} \ge 0$ and $\omega_2 - \frac{\lambda_2}{2}- \frac{\lambda_3}{2} \ge 0$. Hence, the matrix $B$ is nonnegative. All the sums of each row entries are equal to $\lambda_1$, and hence  the matrix $B$ is symmetric stochastic.
The characteristic polynomial $\Ch_3(\lambda)$ of $B$ is factored as
\[
 \Ch_3(\lambda)=\det(B-\lambda I)
 =(\lambda_1-\lambda)(\lambda_2-\lambda)
(\lambda_3-\lambda),
 \]
where we use conditions $(iii)$ and $(iv),$ and therefore, the eigenvalues of $B$ are $\Lambda$. Therefore, we have a symmetric stochastic nonnegative matrix with eigenvalues $\Lambda$ and diagonal entries $\Omega$, hence this concludes statement \ref{enum:sds1}.

\item[\ref{enum:sds4}.] Here, we prove statement \ref{enum:sds4}. From statement \ref{enum:sds1},  statement \ref{enum:sds4} is equivalent to the following: for a given list $\Lambda$, there is a list $\Omega$ with conditions $(i)$--$(iv)$ of statement \ref{enum:sds1} if and only if $\Lambda$ satisfies condition $2\lambda_1 + \lambda_2 + 3\lambda_3 \ge 0.$ Hence, we give a proof for this statement.

    Suppose that there is a list $\Omega$ with conditions $(i)$--$(iv)$ of statement \ref{enum:sds1}.
    Let $ p= 2\lambda_1 + \lambda_2 + 3\lambda_3 $ and $ q= 2\lambda_1 + 3\lambda_2 + \lambda_3$. Then their sum is $p+q = 4(\lambda_1 + \lambda_2 + \lambda_3)$, and by conditions $(i)$ and $(ii)$ of statement \ref{enum:sds1}, it is nonnegative. Their product is
    \begin{align*}
    p\,q= & (2\lambda_1 + \lambda_2 + 3\lambda_3) (2\lambda_1 + 3\lambda_2 + \lambda_3) \\
    = & (2\lambda_1 + \lambda_2 + 3\lambda_3) (2\lambda_1 + 3\lambda_2 + \lambda_3) \\
    &\qquad -12\left((\omega_3-\frac{\lambda_2}{2}-\frac{\lambda_3}{2})^2
     -(\lambda_1-\omega_1)(\lambda_1-\omega_2)
      +\frac{1}{3}(\lambda_1-\lambda_2)(\lambda_1-\lambda_3)\right) \\
    &\qquad +12(\lambda_1+\omega_3) (\omega_1+\omega_2 +\omega_3 - \lambda_1 - \lambda_2- \lambda_3) \\
    = & 12(\omega_1 \omega_2 + \omega_1 \omega_3 + \omega_2 \omega_3 ),
    \end{align*}
    where we use conditions $(ii)$ and $(iv)$. Each of $\omega_1, \omega_2$, and $\omega_3$ is nonnegative, hence the product $p\, q$ is nonnegative. Since the sum and product of reals $p$ and $q$ are nonnegative, the values $p$ and $q$ are nonnegative, hence  we have $ 2\lambda_1 + \lambda_2 + 3\lambda_3 \ge 0 $.

    Conversely, we suppose that $2\lambda_1 + \lambda_2 + 3\lambda_3 \ge 0$. We let $$\omega_1=\frac{1}{3}( \lambda_1+ 2 \lambda_2) \text{ and }\,
    \omega_2= \omega_3=\frac{1}{6}(2 \lambda_1+ \lambda_2 +3 \lambda_3) .$$
    Then, a direct computation shows that these $\omega_1, \omega_2$, and $\omega_3$ satisfy conditions $(i)$--$(iv)$. Therefore, statement \ref{enum:sds4} holds.

\item[\ref{enum:sds3}.] We need the following lemma to prove statement \ref{enum:sds3}.

\begin{lemma}\label{lem:region}
 Let $R$ be the region on the plane with $\lambda_2$- and $\lambda_3$- axes, bounded by
 \[ -\frac{\lambda_1}{2} \le \lambda_2 \le \lambda_1 \ \text{and} \ -\frac{1}{3}(2\lambda_1+ \lambda_2) \le \lambda_3 \le \lambda_2. \]
 Let $R_1$ be the subregion of $R$ bounded by
 \[ -\frac{\lambda_1}{2} \le \lambda_2 \le \lambda_1 \ \text{and} \ \max\{-\frac{\lambda_1}{2}, -2 \lambda_1+3 \lambda_2 \} \le \lambda_3 \le \lambda_2, \]
 $R_2$ be the subregion of $R$ bounded by
 \[-\frac{\lambda_1}{2} \le \lambda_2 \le \lambda_1 \ \text{and} \ -\frac{1}{3}(2\lambda_1+ \lambda_2) \le \lambda_3 \le \min\{ -\frac{\lambda_1}{2} , -\lambda_2\}, \]
 and $R_3$ be the subregion of $R$ bounded by
 \[\frac{\lambda_1}{2} \le \lambda_2 \le \lambda_1 \ \text{and} \ - \lambda_2 \le \lambda_3 \le -2 \lambda_1+3 \lambda_2. \]
 Then, $L_1=\max\{ L_1, L_2, L_3 \}$ on the subregion $R_1$, $L_2=\max\{ L_1, L_2, L_3 \}$ on the subregion $R_2$, and $L_3=\max\{ L_1, L_2, L_3 \}$ on the subregion $R_3$.
\end{lemma}

\begin{proof}
Expression $L_1$ is well defined on the whole region $R$.  Expression $L_2$ on $R$ is well defined on the subregion
 $-(\lambda_1 +2\lambda_2)(\lambda_1 +2 \lambda_3) \ge 0$ or equivalently
 \[ -\frac{\lambda_1}{2} \le \lambda_2 \le \lambda_1 \ \text{and} \ \lambda_3 \le -\frac{\lambda_1}{2}.
 \]
 It is a subregion of $R_2 \cup R_3$. Expression $L_3$ on $R$ is well defined on the subregion
 $-(2\lambda_1+\lambda_2-3\lambda_3) (2\lambda_1-3\lambda_2 +\lambda_3) \ge 0$ or, equivalently,
 \[ \lambda_2 \le \lambda_1 \ \text{and} \ \lambda_3 \le -2 \lambda_1+ 3 \lambda_2,
 \]
 a subregion of $R_2 \cup R_3$. On region $R_1$, only $L_1$ is well defined, and therefore,
 $L_1=\max\{ L_1, L_2, L_3 \}$ on $R_1$.
 On region $R_2$,
 \[ L_2 - L_1 = \frac{1}{6}(\lambda_1+ 2\lambda_3) +\frac{1}{2\sqrt{3}}\sqrt{-(\lambda_1 +2\lambda_2)(\lambda_1 +2 \lambda_3)} \ge 0,
 \]
 therefore $L_2 \ge L_1$ on $R_2$. Let $r_2$ be the region bounded by
 \[ -\frac{2}{5} \lambda_1 \le \lambda_2 \le \lambda_1 \ \text{and} \ -\frac{1}{3}(2 \lambda_1 +\lambda_2) \le \lambda_3 \le \min\{ -2 \lambda_1+3 \lambda_2, -\lambda_2 \}.
 \]
 Then, $r_2$ is a subregion of $R_2$ and on $R_2$, $L_3$ is well defined on $r_2$. On $r_2$, $L_2-L_3 \ge 0$ by a direct computation. Hence, $L_2=\max\{ L_1, L_2, L_3 \}$ on $R_2$.
 On region $R_3$,
  \begin{equation*}
L_3 - L_1
= \frac{1}{12}(2\lambda_1-3\lambda_2+ \lambda_3) +\frac{1}{4\sqrt{3}}\sqrt{-(2\lambda_1+\lambda_2-3\lambda_3) (2\lambda_1-3\lambda_2 +\lambda_3)}
 \ge 0,
\end{equation*}
 therefore $L_3 \ge L_1$ on $R_3$. Let $r_3$ be the region bounded by
 \[ \frac{\lambda_1}{2} \le \lambda_2 \le \lambda_1 \ \text{and} \ -\lambda_2 \le \lambda_3 \le -\frac{\lambda_1}{2} .
 \]
 Then, $r_3$ is a subregion of $R_3$ and on $R_3$, $L_2$ is well defined on $r_3$. On $r_3$, $L_3-L_2 \ge 0$ by a direct computation. Hence, $L_3=\max\{ L_1, L_2, L_3 \}$ on $R_3$.
\end{proof}

 We now prove statement \ref{enum:sds3}. Assuming that statement \ref{enum:sds1} holds, we solve two equations $(ii)$ and $(iv)$, and then, for $\omega_2$ and $\omega_3$ we have the values (\ref{eq:sdsw2}) and (\ref{eq:sdsw3}).
  We first show that if the number $\omega_1$ is within range (\ref{eq:sdsw1}), then taking numbers $\omega_2$ and $\omega_3$ with expression (\ref{eq:sdsw2}) and (\ref{eq:sdsw3}), $\Omega$ satisfies conditions $(i)$--$(iv)$. Second, we show that when $\omega_1$ is out of range (\ref{eq:sdsw1}), $\Omega$ does not satisfy them.

  Suppose that list $\Omega$ satisfies range (\ref{eq:sdsw1}) and expressions (\ref{eq:sdsw2}) and (\ref{eq:sdsw3}). Then, we obtain $\omega_1+\omega_2 +\omega_3 = \lambda_1 + \lambda_2+ \lambda_3$ and
    \[
    (\omega_3-\frac{\lambda_2}{2}-\frac{\lambda_3}{2})^2
    =(\lambda_1-\omega_1)(\lambda_1-\omega_2) -\frac{1}{3}(\lambda_1-\lambda_2)(\lambda_1-\lambda_3)
    \]
    by a direct computation. Hence, conditions $(ii)$ and $(iv)$ hold.

    We now show the inequality $\omega_1 \ge\omega_2\ge \omega_3 \ge 0$. From the range of $\omega_1$ in (\ref{eq:sdsw1}), we have
    \[ \frac{1}{6} (2 \lambda_1 +3\lambda_2 +\lambda_3) \le \omega_1 \le \frac{1}{3}(\lambda_1+2 \lambda_2).
    \]
    Using the value $\omega_2$, we obtain $0 \le \omega_1 - \omega_2 \le \frac{1}{2}(\lambda_2 -\lambda_3) $, hence $\omega_1 \ge\omega_2$. In the same manner, using the value $\omega_2$, we obtain $0 \le \omega_2 - \omega_3 \le \frac{1}{2}(\lambda_2 -\lambda_3) $, hence $\omega_2 \ge\omega_3$. Using the range of $\omega_1$ and the value $\omega_3$, we have $\frac{1}{6} (2 \lambda_1 +\lambda_2 +3\lambda_3) \le \omega_3 \le \frac{1}{3}(\lambda_1+2 \lambda_2)$. From statement \ref{enum:sds4}, $2 \lambda_1 +\lambda_2 +3\lambda_3 \ge 0$ implies $\omega_3 \ge 0$, hence condition $(i)$ holds.

    Now, we show the inequality $\omega_3 - \frac{\lambda_2}{2}- \frac{\lambda_3}{2} \ge 0$. On region $R_1$, by Lemma \ref{lem:region} we have
    \[ L_1 \le \omega_1 \le \frac{1}{3}(\lambda_1+2 \lambda_2). \]
    Therefore,
    \[ \frac{1}{6} (2 \lambda_1 -3\lambda_2 +\lambda_3) \le \omega_3 - \frac{\lambda_2}{2}- \frac{\lambda_3}{2}  \le \frac{1}{3}(\lambda_1- \lambda_2). \]
    Since $-2 \lambda_1 +3 \lambda_2 \le \lambda_3$ on $R_1$, we have $\omega_3 - \frac{\lambda_2}{2}- \frac{\lambda_3}{2} \ge 0$ on $R_1$.
    On region $R_2$, by Lemma \ref{lem:region} we have
    \[ L_2 \le \omega_1 \le \frac{1}{3}(\lambda_1+2 \lambda_2).
    \]
    Then,
    \[ \frac{1}{12} (u - v ) \le \omega_3 - \frac{\lambda_2}{2}- \frac{\lambda_3}{2} \le \frac{1}{3}(\lambda_1-\lambda_2),\]
    where
\begin{align*}
     u= & 3( \lambda_1 -  \lambda_2 -  \lambda_3) -l, \\
     v= & \sqrt{-3( \lambda_1 +3 \lambda_2 -  \lambda_3 +l)( \lambda_1 -  \lambda_2 +3  \lambda_3 +l)}, \text{ and } \\
    l= & \sqrt{-3(\lambda_1 +2\lambda_2)(\lambda_1 +2 \lambda_3)}.
 \end{align*}
    We may easily check that $u$ and $v$ are nonnegative on $R_2$. Hence, $u-v \ge 0$ if and only if $u^2 - v^2 \ge 0 $. By a direct computation we have
    \[ u^2 -v^2 = 12(\lambda_2 +\lambda_3)(-3 \lambda_1 + l).
    \]
    On $R_2$ we have  $\lambda_2 +\lambda_3 \le 0$. In addition, we have $-3 \lambda_1 + l \le 0$, because
    \[ -(3 \lambda_1)^2 + l^2 = -6( \lambda_1(2 \lambda_1+\lambda_2 +3 \lambda_3)- 2 \lambda_3(\lambda_1-\lambda_2)) \le 0,
    \]
    where we use $2 \lambda_1+\lambda_2 +3 \lambda_3 \ge 0$, $\lambda_1 \ge 0$, $\lambda_1 - \lambda_2 \ge 0$, and $ \lambda_3 \le 0$ on $R_2$. Therefore, $u^2 -v^2 \ge 0$ and therefore, $\omega_3 - \frac{\lambda_2}{2}- \frac{\lambda_3}{2} \ge 0$ on $R_2$.
    On region $R_3$, by Lemma \ref{lem:region} we have
    \[ L_3 \le \omega_1 \le \frac{1}{3}(\lambda_1+2 \lambda_2).
    \]
    Then,
    \[ \frac{1}{24} (u - v ) \le \omega_3 - \frac{\lambda_2}{2} - \frac{\lambda_3}{2} \le \frac{1}{3}(\lambda_1- \lambda_2),\]
    where
 \begin{align*}
    u= & 6 \lambda_1 - 3 \lambda_2 - 3 \lambda_3 -l, \\
    v= & \sqrt{- 3(2\lambda_1 +3 \lambda_2 -  5\lambda_3 +l)( 2\lambda_1 -  5\lambda_2 +3  \lambda_3 +l)}, \text{ and }\\
    l= &  \sqrt{-3(2\lambda_1+\lambda_2-3\lambda_3) (2\lambda_1-3\lambda_2 +\lambda_3)}.
 \end{align*}
    We may check that $u$ and $v$ are nonnegative on $R_3$ and $u^2 =v^2$. Hence, $ \omega_3 - \frac{\lambda_2}{2}- \frac{\lambda_3}{2} \ge 0$ on $R_3$. Therefore $\Omega$ satisfies conditions $(i)$--$(iv)$.

 Now, consider the other case when the number $\omega_1$ is out of range (\ref{eq:sdsw1}). Assume that $\omega_1 >\frac{1}{3}(\lambda_1+2 \lambda_2)$. Then, the inside of the square root part of $\omega_2$ is negative, hence $\omega_2$ is not real, violating the assumption that $\omega_2$ is real.

    Consider region  $R_1$ and assume that $\omega_1 < L_1$. For $\frac{1}{3}(\lambda_1+2 \lambda_3) \le \omega_1 < L_1$, a direct computation gives $\omega_1 < \omega_2$, violating condition $\omega_1 \ge \omega_2$.
    For $\omega_1 < \frac{1}{3}(\lambda_1+2 \lambda_3)$, $\omega_2$ is not real, violating the assumption that $\omega_2$ is real.

    Consider region  $R_2$ and assume that $\omega_1 < L_2$. For $L_1 \le \omega_1 < L_2$, a direct computation gives $\omega_3 <0$, violating condition $\omega_3 \ge 0$.
    For $\frac{1}{3}(\lambda_1+2 \lambda_3) \le \omega_1 < L_1$, we have $\omega_1 < \omega_2$, violating condition $\omega_1 \ge \omega_2$.
    For $\omega_1 < \frac{1}{3}(\lambda_1+2 \lambda_3)$, we have that $\omega_2$ is not real, violating the assumption that $\omega_2$ is real.

    Consider region $R_3$ and assume that $\omega_1 < L_3$. For $L_1 \le \omega_1 < L_3$, a direct computation gives $\omega_3 - \frac{\lambda_2}{2}- \frac{\lambda_3}{2} < 0$, violating condition $\omega_3 - \frac{\lambda_2}{2}- \frac{\lambda_3}{2} \ge 0$.
    For $\frac{1}{3}(\lambda_1+2 \lambda_3) \le \omega_1 < L_1$, we have $\omega_1 < \omega_2$, violating condition $\omega_1 \ge \omega_2$.
    For $\omega_1 < \frac{1}{3}(\lambda_1+2 \lambda_3)$, we have that $\omega_2$ is not real, violating the pre-assumption that $\omega_2$ is real. Therefore if $\omega_1$ is out of range (\ref{eq:sdsw1}), then $\Omega$ does not satisfy conditions $(i)$--$(iv)$.
    Hence, we conclude statement \ref{enum:sds3}.
\end{enumerate}
\end{proof}

%
%
%

\section{Doubly stochastic $3 \times 3$ nonnegative matrices} \label{sec:Doubly-stochas}

We discuss a double stochastic $3 \times 3$ nonnegative matrix case with prescribed real or complex eigenvalues and diagonal entries. Consider first the case of real eigenvalues. Statement \ref{enum:rds4} in the following Theorem \ref{the:rds} was proved by Perfect and Minsky \cite{PerfectMirsky65} in 1965.

\begin{theorem}[real, doubly stochastic case] \label{the:rds}
  Let $\Lambda=\{\lambda_1, \lambda_2, \lambda_3\}$ and $\Omega=\{\omega_1, \omega_2, \omega_3\}$
be lists of real numbers with $\lambda_1 \ge \lambda_2 \ge\lambda_3$ and $\omega_1 \ge\omega_2\ge \omega_3$.

 \begin{enumerate}
 \item \label{enum:rds1} There is a doubly stochastic nonnegative matrix with eigenvalues $\Lambda$ and diagonal entries $\Omega$ if and only if the lists $\Lambda$ and $\Omega$ satisfy
 \begin{description}
     \item [] (i) $\omega_i \ge 0$ for $i= $ 1, 2, and 3,
    \item [] (ii) $\omega_1+\omega_2 +\omega_3 = \lambda_1 + \lambda_2+ \lambda_3$,
    \item [] (iii) $\omega_3 - \frac{\lambda_2}{2}- \frac{\lambda_3}{2} \ge 0$,
    \item [] (iv) $(\lambda_1-\omega_1)(\lambda_1-\omega_2) -\frac{1}{3}(\lambda_1-\lambda_2)(\lambda_1-\lambda_3) \ge 0$,
    \item [] (v) $(\omega_3-\frac{\lambda_2}{2}-\frac{\lambda_3}{2})^2 \ge
    (\lambda_1-\omega_1)(\lambda_1-\omega_2) -\frac{1}{3}(\lambda_1-\lambda_2)(\lambda_1-\lambda_3)$.
 \end{description}

\item \label{enum:rds3} If $\Lambda$ and $\Omega$ satisfy the conditions of statement \ref{enum:rds1},
then the exact range for $\omega_1$ is
     \begin{equation}\label{eq:rdsw1}
        \max\{ L_1, L_2, L_3 \} \le \omega_1 \le \min\{\lambda_1 + \lambda_2+\lambda_3, U_2 \},
     \end{equation}
  where $ L_1, L_2$ and $L_3$ are defined in (\ref{eq:Li}) and
  \begin{equation} \label{eq:U2}
        U_2 =\frac{1}{2}(\lambda_2 +\lambda_3)+ \frac{1}{2\sqrt{3}} \sqrt{4(\lambda_1-\lambda_2)(\lambda_1 -\lambda_3)+3(\lambda_2 -\lambda_3)^2}.
  \end{equation}

  For the values for $\omega_2$ and $\omega_3$, we may take
  \begin{equation}\label{eq:rdsw2}
  \omega_2=
    \begin{cases}
        \frac{1}{2}(\lambda_1 +\lambda_2 +\lambda_3 -\omega_1)& \\
        \qquad +\frac{1}{2\sqrt{3}} \sqrt{-
        (3\omega_1 -\lambda_1 -2\lambda_2)
        (3\omega_1 -\lambda_1 -2\lambda_3)},
        &\text{if }\ \omega_1\le \frac{1}{3}(\lambda_1 +2 \lambda_2),\\
        \frac{1}{2}(\lambda_1 +\lambda_2 +\lambda_3 -\omega_1), &\text{otherwise },
    \end{cases}
  \end{equation}
  and
   \begin{equation}\label{eq:rdsw3}
  \omega_3=
    \begin{cases}
        \frac{1}{2}(\lambda_1 +\lambda_2 +\lambda_3 -\omega_1)& \\
        \qquad -\frac{1}{2\sqrt{3}} \sqrt{-
        (3\omega_1 -\lambda_1 -2\lambda_2)
        (3\omega_1 -\lambda_1 -2\lambda_3)},
        &\text{if }\ \omega_1\le \frac{1}{3}(\lambda_1 +2 \lambda_2),\\
        \frac{1}{2}(\lambda_1 +\lambda_2 +\lambda_3 -\omega_1), &\text{otherwise. }
    \end{cases}
  \end{equation}

\item \cite{PerfectMirsky65} \label{enum:rds4} The list $\Lambda$ is realizable as the eigenvalues of a doubly stochastic nonnegative matrix if and only if the following holds:
    \[ 2\lambda_1 + \lambda_2 + 3\lambda_3 \ge 0.
    \]

  \end{enumerate}
\end{theorem}

\label{enum:rds2}  For $\Lambda$ and $\Omega$ satisfying the conditions of statement \ref{enum:rds1}, we find the following doubly stochastic nonnegative matrix
  \[
  \left(\begin{array}{cccccccc}
  \omega_1 & \omega_3-\frac{\lambda_2}{2}-\frac{\lambda_3}{2} +\sqrt{W} & \omega_2-\frac{\lambda_2}{2}-\frac{\lambda_3}{2} -\sqrt{W} \\
  \omega_3-\frac{\lambda_2}{2}-\frac{\lambda_3}{2}-\sqrt{W} & \omega_2 & \omega_1-\frac{\lambda_2}{2}-\frac{\lambda_3}{2} +\sqrt{W}\\
  \omega_2-\frac{\lambda_2}{2}-\frac{\lambda_3}{2}+\sqrt{W}& \omega_1-\frac{\lambda_2}{2}-\frac{\lambda_3}{2} -\sqrt{W} &  \omega_3
  \end{array}\right),
  \]
  where
\begin{equation}\label{eq:w}
    W=(\omega_3-\frac{\lambda_2}{2}-\frac{\lambda_3}{2})^2
    -(\lambda_1-\omega_1)(\lambda_1-\omega_2) +\frac{1}{3}(\lambda_1-\lambda_2)(\lambda_1-\lambda_3),
\end{equation}
  with eigenvalues $\Lambda$ and diagonal entries $\Omega$.
For a usual stochastic matrix, we may take $\lambda_1 = 1$.

\begin{proof}
\begin{enumerate}
\item
 Suppose that the list $\Lambda$ is realizable as the eigenvalues of a doubly stochastic nonnegative matrix $A$ with diagonal entries $\Omega$. Then, the elements $\omega_1$, $\omega_2$, and $\omega_3$ are nonnegative, hence we have condition $(i)$.
 The characteristic polynomial $\Ch_1(\lambda)$ of $A$, using the eigenvalues, is
 \[ \Ch_1(\lambda) = \lambda_1 \lambda_2 \lambda_3 -(\lambda_1 \lambda_2+\lambda_1 \lambda_3 +\lambda_2 \lambda_3)\lambda+ (\lambda_1 + \lambda_2 + \lambda_3)\lambda^2-\lambda^3. \]

  We may write the doubly stochastic nonnegative matrix $A$ without loss of generality as
  \[ A=\left(\begin{array}{cccccccc}
    \omega_1 & s & \lambda_1 - \omega_1 -s \\
    t  & \omega_2 & \lambda_1 - \omega_2 -t  \\
    \lambda_1 - \omega_1 -t  & \lambda_1 - \omega_2 -s  & \omega_3
             \end{array}\right),
  \]
 where each entry is nonnegative and $\omega_3=s+t +\omega_1+\omega_2-\lambda_1$.
 From this matrix we have another form of the characteristic polynomial $\Ch_2(\lambda)$ of $A$, which is
  \begin{multline*}
   \Ch_2(\lambda) = (-3s\, t +(s+t-\omega_1 -\omega_2) \lambda_1 +3 \omega_1 \omega_2) \lambda_1 \\
   -(-3s\, t  +(2 s +2t  +\omega_1  + \omega_2  -2 \lambda_1)\lambda_1  +3 \omega_1 \omega_2 )\lambda\\
   + (\omega_1 + \omega_2  + \omega_3 )\lambda^2
   - \lambda^3.
  \end{multline*}
Then, the two characteristic polynomials $\Ch_1(\lambda)$ and $\Ch_2(\lambda)$ are identical. Comparing the coefficients of $\lambda^2$, we have
\[\lambda_1 + \lambda_2 +\lambda_3  = \omega_1+ \omega_2+ \omega_3, \]
which is condition $(ii)$.
From $\omega_3=s+t +\omega_1+\omega_2-\lambda_1$ we have
\begin{equation*}\
    \frac{s+t}{2}=\omega_3-\frac{\lambda_2}{2}-\frac{\lambda_3}{2},
\end{equation*}
which is nonnegative because $s$ and $t$ are nonnegative, hence we have $(iii)$.
Comparing the coefficients of $\lambda$, we have
\[\lambda_1 \lambda_2+\lambda_1 \lambda_3+\lambda_2 \lambda_3 = -3s\, t  +(2 s +2t  +\omega_1  + \omega_2  -2 \lambda_1)\lambda_1  +3 \omega_1 \omega_2. \]
Then,
\begin{equation*}
    s\,t = (\lambda_1-\omega_1)(\lambda_1-\omega_2) -\frac{1}{3}(\lambda_1-\lambda_2)(\lambda_1-\lambda_3),
\end{equation*}
which is nonnegative, hence condition $(iv)$ holds.
By the inequality of arithmetic and geometric means, we have
$ (\frac{s+t}{2})^2-s\,t \ge 0$, which is condition $(v)$.
Therefore all conditions $(i)$--$(v)$ are satisfied.

Now, suppose that the lists $\Lambda$ and $\Omega$ satisfy conditions $(i)$--$(v)$.
Let $B$ be the following matrix
\[ B=
  \left(\begin{array}{cccccccc}
  \omega_1 & \omega_3-\frac{\lambda_2}{2}-\frac{\lambda_3}{2} +\sqrt{W} & \omega_2-\frac{\lambda_2}{2}-\frac{\lambda_3}{2} -\sqrt{W} \\
  \omega_3-\frac{\lambda_2}{2}-\frac{\lambda_3}{2}-\sqrt{W} & \omega_2 & \omega_1-\frac{\lambda_2}{2}-\frac{\lambda_3}{2} +\sqrt{W}\\
  \omega_2-\frac{\lambda_2}{2}-\frac{\lambda_3}{2}+\sqrt{W}& \omega_1-\frac{\lambda_2}{2}-\frac{\lambda_3}{2} -\sqrt{W} &  \omega_3
  \end{array}\right),
  \]
  where $W$ is defined in (\ref{eq:w}).
By condition $(i)$, $\omega_i$ is nonnegative for $i= $ 1, 2, and 3.
Let $s=\omega_3-\frac{\lambda_2}{2}-\frac{\lambda_3}{2}+\sqrt{W}$ and $t=\omega_3-\frac{\lambda_2}{2}-\frac{\lambda_3}{2}-\sqrt{W}$.
The number $W$ is nonnegative by condition $(v)$, hence $s$ and $t$ are real numbers.
Since the sum $s+t$ and product $s\,t$ are nonnegative by conditions $(iii)$ and $(iv)$, respectively, we have that $s$ and $t$ are nonnegative.

Because $\omega_3-\frac{\lambda_2}{2}-\frac{\lambda_3}{2}-\sqrt{W} \ge 0$ and $\omega_1 \ge\omega_2\ge \omega_3$ by pre-assumption, we have that $\omega_1 - \frac{\lambda_2}{2}- \frac{\lambda_3}{2} -\sqrt{W}\ge 0$ and $\omega_2 - \frac{\lambda_2}{2}- \frac{\lambda_3}{2} -\sqrt{W}\ge 0$. Hence, the matrix $B$ is nonnegative. All sums of each row and column entries are equal to $\lambda_1$, hence the matrix $B$ is doubly stochastic.
The characteristic polynomial $\Ch_3(\lambda)$ of $B$ is factored as
\[
 \Ch_3(\lambda)=\det(B-\lambda I)
 =(\lambda_1-\lambda)(\lambda_2-\lambda)
(\lambda_3-\lambda),
 \]
where we use conditions $(ii)$, and therefore, the eigenvalues of $B$ are $\Lambda$. Therefore, we obtain a doubly stochastic nonnegative matrix with eigenvalues $\Lambda$ and diagonal entries $\Omega$, hence this concludes statement \ref{enum:rds1}.

\item[\ref{enum:rds4}.] The proof of statement \ref{enum:rds4} is similar to that of statement \ref{enum:sds4} of Theorem \ref{the:sds}, hence we omit it.

\item[\ref{enum:rds3}.] We need the following lemma to prove statement \ref{enum:rds3}.

\begin{lemma}\label{lem:regionQ}
 Let $R$ be the region on the plane with $\lambda_2$- and $\lambda_3$- axes, bounded by
 \[ -\frac{\lambda_1}{2} \le \lambda_2 \le \lambda_1 \ \text{and} \ -\frac{1}{3}(2\lambda_1+ \lambda_2) \le \lambda_3 \le \lambda_2. \]
 Let $Q_1$ be the subregion of $R$ bounded by
 \[ -\frac{\lambda_1}{2} \le \lambda_2 \le \lambda_1 \ \text{and} \ -\frac{1}{3}(2\lambda_1+\lambda_2) \le \lambda_3 \le \min\{\lambda_2,
 -\frac{\lambda_1^2+2 \lambda_1 \lambda_2}
 {2 \lambda_1 + \lambda_2} \}, \]
and $Q_2$ the subregion of $R$ bounded by
 \[(-2+\sqrt{3})\lambda_1 \le \lambda_2 \le \lambda_1 \ \text{and} \ -\frac{\lambda_1^2+2 \lambda_1 \lambda_2}
 {2 \lambda_1 + \lambda_2} \le \lambda_3 \le \lambda_2. \]
 Then, $\lambda_1+\lambda_2+\lambda_3=
 \min\{\lambda_1+\lambda_2+\lambda_3,\ U_2\}$ on subregion $Q_1$ and $U_2=\min\{ \lambda_1+\lambda_2+\lambda_3, U_2 \}$ on subregion $Q_2$.
\end{lemma}

\begin{proof}
 We have
\begin{equation*}
\lambda_1+\lambda_2+\lambda_3 -U_2
= \lambda_1 + \frac{\lambda_2}{2} + \frac{\lambda_3}{2} - \frac{1}{2\sqrt{3}} \sqrt{4(\lambda_1-\lambda_2)(\lambda_1 -\lambda_3)+3(\lambda_2 -\lambda_3)^2}.
\end{equation*}
Since $ \lambda_1 + \frac{\lambda_2}{2} + \frac{\lambda_3}{2} \ge 0$, we consider that
\begin{multline*}
\left(\lambda_1 + \frac{\lambda_2}{2} + \frac{\lambda_3}{2}\right)^2 - \left(\frac{1}{2\sqrt{3}} \sqrt{4(\lambda_1-\lambda_2)(\lambda_1 -\lambda_3)+3(\lambda_2 -\lambda_3)^2}\right)^2 \\
= \frac{2}{3} (\lambda_1^2+2 \lambda_1 \lambda_2+ 2 \lambda_1 \lambda_3 + \lambda_2 \lambda_3).
\end{multline*}
On $Q_1$, we have $\lambda_1^2+2 \lambda_1 \lambda_2+ 2 \lambda_1 \lambda_3 + \lambda_2 \lambda_3 \le 0$, hence
$\lambda_1+\lambda_2+\lambda_3 \le U_2$, which says that $\lambda_1+\lambda_2+\lambda_3=
 \min\{\lambda_1+\lambda_2+\lambda_3,\ U_2\}$ on $Q_1$. On $Q_2$, we have $\lambda_1^2+2 \lambda_1 \lambda_2+ 2 \lambda_1 \lambda_3 + \lambda_2 \lambda_3 \ge 0$, that is,
$\lambda_1+\lambda_2+\lambda_3 \ge U_2$, which says that $U_2=
 \min\{\lambda_1+\lambda_2+\lambda_3,\ U_2\}$ on $Q_2$.
\end{proof}

We assume that statement \ref{enum:rds1} holds. Then, from statement \ref{enum:rds4}, we have $ 2\lambda_1 + \lambda_2 + 3\lambda_3 \ge 0.$
    We first show that if the number $\omega_1$ is within range (\ref{eq:rdsw1}), then taking numbers $\omega_2$ and $\omega_3$ with expression (\ref{eq:rdsw2}) and (\ref{eq:rdsw3}), the list $\Omega$ satisfies conditions $(i)$--$(iv)$. Second, we show that when $\omega_1$ is out of range (\ref{eq:rdsw1}), $\Omega$ does not satisfy them.

    Suppose that $\Omega$ is within ranges (\ref{eq:rdsw1}) $-$
 (\ref{eq:rdsw3}). We split the range of $\omega_1$ into two subranges.
 Assume first that $\omega_1$ is in the subrange
 \[ \max \{ L_1, L_2, L_3 \} \le \omega_1 \le
    \frac{1}{3}(\lambda_1+2 \lambda_2), \]
 and then take
 \[ \omega_2=\frac{1}{2}(\lambda_1 +\lambda_2 +\lambda_3 -\omega_1)+\frac{1}{2\sqrt{3}} \sqrt{-
        (3\omega_1 -\lambda_1 -2\lambda_2)
        (3\omega_1 -\lambda_1 -2\lambda_3)}, \]
        and
\[ \omega_3=\frac{1}{2}(\lambda_1 +\lambda_2 +\lambda_3 -\omega_1) -\frac{1}{2\sqrt{3}} \sqrt{-
        (3\omega_1 -\lambda_1 -2\lambda_2)
        (3\omega_1 -\lambda_1 -2\lambda_3)}. \]
Then, the list $\Omega$ satisfies conditions
$(i)$--$(iv)$, which follows from the proof of statement \ref{enum:sds3} of Theorem \ref{the:sds}.

Now, assume the other subrange
\[ \frac{1}{3}(\lambda_1+2 \lambda_2) \le \omega_1 \le
    \min\{ \lambda_1 + \lambda_2+\lambda_3, U_2 \}, \]
 and then take
 \[ \omega_2=\omega_3=\frac{1}{2}(\lambda_1 +\lambda_2 +\lambda_3 -\omega_1). \]

From this assumption we can check $\omega_1+\omega_2 +\omega_3 = \lambda_1 + \lambda_2+ \lambda_3$, which is condition $(ii)$.

Consider the list $\Lambda$ on region $Q_1$. On $Q_1$, by Lemma \ref{lem:regionQ}, $\lambda_1 + \lambda_2+\lambda_3
=\min\{ \lambda_1 + \lambda_2+\lambda_3, U_2 \}$, hence $\frac{1}{3}(\lambda_1+2 \lambda_2) \le \omega_1 \le \lambda_1 +\lambda_2 +\lambda_3 $.
Since $\omega_1 - \omega_2= \frac{1}{2}(3 \omega_1 - \lambda_1 -\lambda_2 -\lambda_3)$, $\frac{1}{2}(\lambda_2 -\lambda_3) \le \omega_1 - \omega_2 \le \lambda_1 +\lambda_2 +\lambda_3$. Hence, $\omega_1 \ge \omega_2$.  Since $\omega_3= \frac{1}{2}(\lambda_1 +\lambda_2 +\lambda_3-\omega_1)$,
$0\le \omega_3 \le \frac {1} {6} (2\lambda_1 + \lambda_2 +
   3\lambda_3)$. Therefore, condition $(i)$ holds.

From $0\le \omega_3 \le \frac {1} {6} (2\lambda_1 + \lambda_2 +
   3\lambda_3)$, we have  $-\frac{1}{2}(\lambda_2 + \lambda_3) \le \omega_3 -\frac{\lambda_2}{2} -\frac{\lambda_3}{2}\le \frac{1}{3}(\lambda_1 - \lambda_2)$. However, on $Q_1$,\ $\lambda_3 \le -\frac{\lambda_1^2
+ 2\lambda_1 \lambda_2}{2 \lambda_1+\lambda_2}$, which implies
   \[ -\frac{1}{2}(\lambda_2 + \lambda_3) \ge -\frac{1}{2}(\lambda_2
  - \frac{\lambda_1^2 +2 \lambda_1 \lambda_2}{2 \lambda_1 +\lambda_2}) =
  \frac{\lambda_1^2 - \lambda_2^2}{4 \lambda_1 +2\lambda_2}\ge 0.\]
  Therefore, condition $(iii)$ holds.

Let $V=(\lambda_1-\omega_1)(\lambda_1-\omega_2) -\frac{1}{3}(\lambda_1-\lambda_2)(\lambda_1-\lambda_3)$. Using $\frac{1}{3}(\lambda_1+2 \lambda_2) \le \omega_1 \le \lambda_1 +\lambda_2 +\lambda_3 $ and $w_2=\frac{1}{2}(\lambda_1 +\lambda_2 +\lambda_3 -\omega_1)$, we have that
$ -\frac{1}{3}(\lambda_1^2 + 2 \lambda_1 \lambda_2
+ 2 \lambda_1 \lambda_3 + \lambda_2 \lambda_3) \le V \le \frac{1}{9}(\lambda_1 - \lambda_2)^2$.
On $Q_1$,\ $\lambda_3 \le -\frac{\lambda_1^2
+ 2\lambda_1 \lambda_2}{2 \lambda_1+\lambda_2}$, which implies  $-\frac{1}{3}(\lambda_1^2 + 2 \lambda_1 \lambda_2
+ 2 \lambda_1 \lambda_3 + \lambda_2 \lambda_3) \ge 0$.   Hence, condition $(iv)$ holds.

Now, consider the number $W$ defined in (\ref{eq:w}). Using $\frac{1}{3}(\lambda_1+2 \lambda_2) \le \omega_1 \le \lambda_1 +\lambda_2 +\lambda_3 $ and $w_2=\frac{1}{2}(\lambda_1 +\lambda_2 +\lambda_3 -\omega_1)$, we have
$ 0 \le W \le \frac{1}{12}(2\lambda_1 + 3\lambda_2 + \lambda_3) (2\lambda_1 + \lambda_2 + 3\lambda_3)$. Hence, condition $(v)$ holds.

Next consider the list $\Lambda$ on region $Q_2$. On $Q_2$, by Lemma \ref{lem:regionQ}, we have $U_2=\min\{ \lambda_1 + \lambda_2+\lambda_3, U_2 \}$, hence $\frac{1}{3}(\lambda_1+2 \lambda_2) \le \omega_1 \le U_2 $.
 Let $S= \sqrt{4(\lambda_1-\lambda_2)(\lambda_1 -\lambda_3)+3(\lambda_2 -\lambda_3)^2}$.

Then, we have
$\frac{1}{2}(\lambda_2 -\lambda_3) \le \omega_1 - \omega_2 \le \frac{1}{4}(-2\lambda_1 +\lambda_2 +\lambda_3) + \frac{\sqrt{3}}{4} S$, hence $\omega_1 \ge \omega_2$. Furthermore, we have
$\frac{1}{4}(2\lambda_1 +\lambda_2 +\lambda_3) - \frac{1}{4\sqrt{3}} S \le \omega_3 \le
\frac{1}{6}(2 \lambda_1 +\lambda_2 +3\lambda_3)$.
Consider  $\left(\frac{1}{4}(2\lambda_1 +\lambda_2 +\lambda_3)\right)^2 - \left(\frac{1}{4\sqrt{3}} S\right)^2 =
\frac{1}{6}(\lambda_1^2 + 2 \lambda_1 \lambda_2
+ 2 \lambda_1 \lambda_3 + \lambda_2 \lambda_3)$, which is nonnegative on $Q_2$, hence $\omega_3 \ge 0$. Therefore, condition $(i)$ holds.

A direct computation shows $\frac{1}{4}(2\lambda_1 -\lambda_2 -\lambda_3) - \frac{1}{4\sqrt{3}} S \le \omega_3 -\frac{\lambda_2}{2} -\frac{\lambda_3}{2} \le \frac{1}{3}(\lambda_1 -  \lambda_2) $. On $Q_2,\ \frac{1}{4}(2\lambda_1 -\lambda_2 -\lambda_3) - \frac{1}{4\sqrt{3}} S  >0$, because $\left(\frac{1}{4}(2\lambda_1 -\lambda_2 -\lambda_3)\right)^2 - \left(\frac{1}{4\sqrt{3}} S \right)^2 = \frac{1}{6}(\lambda_1 -\lambda_2)(\lambda_1 -\lambda_3) \ge 0 $. Therefore, condition $(iii)$ holds.

A direct computation shows that on $Q_2$,
$ 0 \le V \le \frac{1}{9}(\lambda_1 - \lambda_2)^2$.
Therefore, condition $(iv)$ holds.

In addition, a direct computation gives
$ 0 \le W \le \frac{1}{3}(\lambda_1 - \lambda_2)(\lambda_1 - \lambda_3)+\frac{1}{8}(\lambda_2 -\lambda_3)^2 -
 \frac{1}{8\sqrt{3}}(2\lambda_1 - \lambda_2 -\lambda_3)S$. Hence, condition $(v)$ satisfies. Therefore, list $\Omega$ satisfies the all conditions $(i)$--$(v)$.

Now, consider the other case when the number $\omega_1$ is out of range (\ref{eq:rdsw1}). Assume first $\omega_1 < \max\{ L_1, L_2, L_3 \}$.
Consider subregion $R_1$. On $R_1$, since $L_1 = \max\{ L_1, L_2, L_3 \}$, $\omega_1 < \max\{ L_1, L_2, L_3 \}$, which implies $\omega_1 < L_1$. We will show that there is no pair $(\omega_1, \omega_2)$ that satisfies conditions $(i)$--$(v)$ and $\omega_1 < L_1$ on region $R_1$. Indeed it is sufficient to show that if a pair $(\omega_1, \omega_2)$ satisfies $(i)$ and $(ii)$ and $\omega_1 < L_1$, then it violates condition $(v)$, i.e., $W <0$.

It is convenient to view a pair $(\omega_1, \omega_2)$ as a point on a $(\omega_1, \omega_2)-$ plane and $\lambda_1,\ \lambda_2$, and $\lambda_3$ as constants.
If a point $(\omega_1, \omega_2)$ satisfies conditions (a) $\omega_1 \ge \omega_2$, (b) $\omega_1 + 2 \omega_2 \ge \lambda_1 + \lambda_2+ \lambda_3$ (which comes from $\omega_2 \ge \omega_3$ and $(ii)$), and (c) $ \omega_1 < L_1$, then $(\omega_1, \omega_2)$ is a point on the polygon of vertices $p_0,\ p_1$, and $p_2$ containing the line segments $p_0 p_1$, and $p_0 p_2$, but not containing the line segment $p_1 p_2$, where
\begin{align*}
p_0 & =\left( \frac{1}{3}(\lambda_1 + \lambda_2+ \lambda_3),
\frac{1}{3}(\lambda_1 + \lambda_2+ \lambda_3)\right),\\
p_1 & = \left( \frac{1}{6}(2\lambda_1 + 3\lambda_2+ \lambda_3),
\frac{1}{6}(2\lambda_1 + 3\lambda_2+ \lambda_3)\right) \ \text{and} \\
p_2 & = \left( \frac{1}{6}(2\lambda_1 + 3\lambda_2+ \lambda_3),
\frac{1}{12}(4\lambda_1 + 3\lambda_2+ 5\lambda_3)\right).
\end{align*}
We may check that the point $p_0$ is the intersection of two lines $\omega_1 + 2 \omega_2 = \lambda_1 + \lambda_2+ \lambda_3$ and $\omega_1 = \omega_2$, the point $p_1$ is the intersection of two lines $\omega_1 = \omega_2$ and $\omega_1 =L_1$, and the point $p_2$ is the intersection of two lines $\omega_1 =L_1$ and $\omega_1 + 2 \omega_2 = \lambda_1 + \lambda_2+ \lambda_3$.

We can view $W$ as a function of variables $\omega_1$ and $\omega_2$, that is,
\begin{equation*}
W(\omega_1,\omega_2)
= (\lambda_1+\frac{\lambda_2}{2}  +\frac{\lambda_3}{2} -\omega_1 -\omega_2)^2
    -(\lambda_1-\omega_1)(\lambda_1-\omega_2) +\frac{1}{3}(\lambda_1-\lambda_2)(\lambda_1-\lambda_3).
\end{equation*}
Then, $W(\omega_1,\omega_2)$ is an elliptic paraboloid which opens upward with minimum at $p_0$, and we have that $W(p_0)=-\frac{1}{2}(\lambda_2-\lambda_3)^2,\ W(p_1)=0$ and $ W(p_2)=-\frac{1}{16}(\lambda_2-\lambda_3)^2$, which are all nonpositive on region $R_1$. Therefore, $W(\omega_1,\omega_2) < 0$ for each point $(\omega_1,\omega_2)$ on this polygon.

Consider subregion $R_2$. On $R_2$, since $L_2 = \max\{ L_1, L_2, L_3 \}$, then $\omega_1 < \max\{ L_1, L_2, L_3 \}$, which implies $\omega_1 < L_2$. We will show that there is no point $(\omega_1, \omega_2)$ that satisfies conditions $(i)$--$(v)$ and $\omega_1 < L_2$ on region $R_2$. Indeed it is enough to show that if $(\omega_1, \omega_2)$ satisfies $(i)$, $(ii)$ and $\omega_1 < L_2$, then it violates condition $(v)$, i.e., $W <0$.

If a point $(\omega_1, \omega_2)$ satisfies conditions (a) $\omega_1 \ge \omega_2$, (b) $\omega_1 + 2 \omega_2 \ge \lambda_1 + \lambda_2+ \lambda_3$, (c) $\omega_1 + \omega_2 \le \lambda_1 + \lambda_2+ \lambda_3$ (which comes from $ \omega_3 \ge 0$ and (ii) ) and (d) $ \omega_1 < L_2$, then $(\omega_1, \omega_2)$ is a point on the polygon of vertices $p_0, p_3, p_4$ and $p_5$ containing the boundary except
the line segment $p_4 p_5$, where
\begin{align*}
p_3 & =\left( \frac{1}{2}(\lambda_1 + \lambda_2+ \lambda_3), \
\frac{1}{2}(\lambda_1 + \lambda_2+ \lambda_3)\right), \\
p_4 & =\left( L_2,
\lambda_1 + \lambda_2+ \lambda_3-L_2 \right) \ \text{and} \\
p_5 & =\left( L_2,
\frac{1}{2}(\lambda_1 + \lambda_2+ \lambda_3-L_2) \right) .
\end{align*}

We may check that the point $p_3$ is the intersection of two lines $\omega_1 = \omega_2$ and $\omega_1 + \omega_2 = \lambda_1 + \lambda_2+ \lambda_3$, the point $p_4$ is the intersection of two lines $\omega_1 + \omega_2 = \lambda_1 + \lambda_2+ \lambda_3$ and $\omega_1 =L_2$, and the point $p_5$ is the intersection of two lines $\omega_1 =L_1$ and $\omega_1 + 2 \omega_2 = \lambda_1 + \lambda_2+ \lambda_3$.

Then, $W(\omega_1,\omega_2)$ is an elliptic paraboloid which opens upward with minimum at $p_0$, and we have that $W(p_0)=-\frac{1}{2}(\lambda_2-\lambda_3)^2,\ W(p_3)=\frac{1}{2}(\lambda_1+2\lambda_2) (\lambda_1+2\lambda_3), \
W(p_4)=0$ and
\begin{multline*}
W(p_5)=-\frac{1}{48}\left((\lambda_1+2\lambda_2)^2
+(\lambda_1+2\lambda_3)^2-(\lambda_2-\lambda_3)^2\right) \\
 + \frac{1}{8\sqrt{3}}(\lambda_1 + \lambda_2+ \lambda_3)
\sqrt{-(\lambda_1+2\lambda_2)(\lambda_1+2\lambda_3)},
\end{multline*}
which are all nonpositive on region $R_2$.  Therefore, $W(\omega_1,\omega_2) <0$ for each point $(\omega_1,\omega_2)$ on this polygon.

Consider subregion $R_3$. On $R_3$, since $L_3 = \max\{ L_1, L_2, L_3 \}$, we have $\omega_1 < \max\{ L_1, L_2, L_3 \}$, which implies $\omega_1 < L_3$. We will show that there is no point $(\omega_1, \omega_2)$ that satisfies conditions $(i)$--$(v)$ and $\omega_1 < L_3$ on region $R_3$. Indeed it is enough to show that if $(\omega_1, \omega_2)$ satisfies $(i)$--$(iii)$ and $\omega_1 < L_3$, then it violates condition $(v)$, i.e., $W <0$.

If a point $(\omega_1, \omega_2)$ satisfies conditions (a) $\omega_1 \ge \omega_2$, (b) $\omega_1 + 2 \omega_2 \ge \lambda_1 + \lambda_2+ \lambda_3$, (c) $\omega_1 + \omega_2 \le \lambda_1 + \frac{\lambda_2}{2}+ \frac{\lambda_3}{2}$ (which comes from (iii) ) and (d) $ \omega_1 < L_3$, then $(\omega_1, \omega_2)$ is a point on the polygon of vertices $p_0, p_6, p_7$, and $p_8$ containing boundary except for the line segment $p_7 p_8$, where
\begin{align*}
p_6 & =\left( \frac{1}{4}(2\lambda_1 + \lambda_2+ \lambda_3), \
\frac{1}{4}(2\lambda_1 + \lambda_2+ \lambda_3)\right), \\
p_7 &= \left( L_3,
\lambda_1 + \frac{\lambda_2}{2}+ \frac{\lambda_3}{2}-L_3 \right) \text{and} \\
p_8 &= \left( L_3,
\frac{1}{2}(\lambda_1 + \lambda_2+ \lambda_3-L_3) \right) .
\end{align*}

We may check that the point $p_6$ is the intersection of two lines $\omega_1 = \omega_2$ and $\omega_1 + \omega_2 = \lambda_1 + \frac{\lambda_2}{2}+ \frac{\lambda_3}{2}$, the point $p_7$ is the intersection of two lines $\omega_1 + \omega_2 = \lambda_1 + \frac{\lambda_2}{2}+ \frac{\lambda_3}{2}$ and $\omega_1 =L_3$, and the point $p_8$ is the intersection of two lines $\omega_1 =L_3$ and $\omega_1 + 2 \omega_2 = \lambda_1 + \lambda_2+ \lambda_3$.

Then, $W(\omega_1,\omega_2)$ is an elliptic paraboloid which opens upward with minimum at $p_0$, and we have that $W(p_0)=-\frac{1}{2}(\lambda_2-\lambda_3)^2,\ W(p_6)=\frac{1}{48}(2\lambda_1 + \lambda_2-3 \lambda_3) (2\lambda_1 -3 \lambda_2+ \lambda_3), \
W(p_7)=0$ and
\begin{multline*} W(p_8)=-\frac{1}{24}(\lambda_1-\lambda_2)
(\lambda_1-\lambda_3)-\frac{1}{32}(\lambda_2-\lambda_3)^2  \\
+\frac{1}{32\sqrt{3}}(2\lambda_1 - \lambda_2- \lambda_3)
\sqrt{-(2\lambda_1 + \lambda_2- 3\lambda_3)
(2\lambda_1 -3 \lambda_2+ \lambda_3)},
\end{multline*}
which are all nonpositive on region $R_3$. Therefore, $W(\omega_1,\omega_2) < 0$ for each point $(\omega_1,\omega_2)$ on this polygon.

Now, suppose  that $\omega_1 > \min\{\lambda_1 + \lambda_2+\lambda_3, U_2 \}$.
From condition $(ii)$ and nonnegativity of $\omega_2$ and $\omega_3$, we  have $\omega_1 \le \lambda_1 + \lambda_2+\lambda_3$. From conditions $\omega_1 + 2 \omega_2 \ge \lambda_1 + \lambda_2+ \lambda_3$, i.e., $ \omega_2 \ge \frac{1}{2}(\lambda_1 + \lambda_2+ \lambda_3 - \omega_1)$ and $(iv)$, we reduce
$ (\lambda_2 + \lambda_3 -U_2) \le \omega_1 \le U_2$. Therefore, we cannot have $\omega_1 > \min\{\lambda_1 + \lambda_2+\lambda_3, U_2 \}$.
Therefore, if $\omega_1$ is out of range (\ref{eq:rdsw1}), then $\Omega$ does not satisfy conditions $(i)$--$(v)$.
Hence, we conclude statement \ref{enum:rds3}.

\end{enumerate}
\end{proof}

%
%
%


Consider the case of complex eigenvalues

\begin{theorem}[complex, doubly stochastic case] \label{the:cds}
  Let $\Lambda=\{\lambda_1, \lambda_2, \lambda_3\}$ be a list of complex numbers with $ \lambda_1 \ge |\lambda_2| \ge |\lambda_3|$ and  $\Omega=\{\omega_1, \omega_2, \omega_3\}$ be a list of real numbers with $\omega_1 \ge\omega_2\ge \omega_3$.

 \begin{enumerate}
 \item \label{enum:cds1} There is a doubly stochastic nonnegative matrix with eigenvalues $\Lambda$ and diagonal entries
$\Omega$ if and only if the lists $\Lambda$ and $\Omega$ satisfy

  \begin{description}
    \item [] (i) $\omega_i \ge 0$ for $i= $ 1, 2, and 3,
    \item [] (ii) $\omega_1+\omega_2 +\omega_3 = \lambda_1 + \lambda_2+ \lambda_3$,
    \item [] (iii) $\omega_3 - \frac{\lambda_2}{2}- \frac{\lambda_3}{2} \ge 0$,
    \item [] (iv) $(\lambda_1-\omega_1)(\lambda_1-\omega_2) -\frac{1}{3}(\lambda_1-\lambda_2)(\lambda_1-\lambda_3) \ge 0$.
    \end{description}

\item \label{enum:cds3}  If $\Lambda$ and $\Omega$ satisfy the conditions of statement \ref{enum:cds1}, the exact range for $\omega_1$ is
    \begin{equation}\label{eq:cdsw1}
        \frac{1}{3}(\lambda_1 + \lambda_2 + \lambda_3) \le \omega_1 \le \min\{\lambda_1 + \lambda_2+\lambda_3, U_2 \},
        \end{equation}
  where $U_2$ is as defined in (\ref{eq:U2}).
  For the values for $\omega_2$ and $\omega_3$, we may take
        \begin{equation}\label{eq:cdsw2}
         \omega_2=\omega_3=\frac{1}{2} (\lambda_1 + \lambda_2 + \lambda_3- \omega_1).
        \end{equation}

\item \cite{PerfectMirsky65} \label{enum:cds4} The list $\Lambda$ is realizable as the eigenvalues of a doubly stochastic nonnegative matrix if and only if the following hold:
    \begin{description}
        \item[] (i) $\lambda_1 \ge 0$ and $\lambda_2 =  \overline{\lambda_3}$,
        \item[] (ii) $\lambda_1 +\lambda_2+\lambda_3 \ge 0$,
        \item[] (iii) $\lambda_1^2 +\lambda_2^2 +\lambda_3^2 \ge  \lambda_1 \lambda_2 +\lambda_1 \lambda_3 + \lambda_2 \lambda_3$.
    \end{description}
\end{enumerate}
\end{theorem}

\label{enum:cds2} For $\Lambda$ and $\Omega$ satisfying the conditions of statement \ref{enum:cds1}, we find the following the doubly stochastic nonnegative matrix
  \[
  \left(\begin{array}{cccccccc}
  \omega_1 & \omega_3-\frac{\lambda_2}{2}-\frac{\lambda_3}{2} +\sqrt{W} & \omega_2-\frac{\lambda_2}{2}-\frac{\lambda_3}{2} -\sqrt{W} \\
  \omega_3-\frac{\lambda_2}{2}-\frac{\lambda_3}{2}-\sqrt{W} & \omega_2 & \omega_1-\frac{\lambda_2}{2}-\frac{\lambda_3}{2} +\sqrt{W}\\
  \omega_2-\frac{\lambda_2}{2}-\frac{\lambda_3}{2}+\sqrt{W}& \omega_1-\frac{\lambda_2}{2}-\frac{\lambda_3}{2} -\sqrt{W} &  \omega_3
  \end{array}\right),
  \]
  where $W$ is defined in (\ref{eq:w}), with eigenvalues $\Lambda$ and diagonal entries $\Omega$.
For a usual stochastic matrix, we may take $\lambda_1 = 1$.

\begin{proof}
\begin{enumerate}
\item This statement is similar to that of Theorem \ref{the:rds}. The only difference is the absence of condition $(v)$ of statement \ref{enum:rds1} of Theorem \ref{the:rds}. In fact, condition $(v)$ is always true in this case. We may rewrite $W$ as
    \[ W= \frac{3}{4}\left(\omega_1+\omega_2 -\frac{2}{3}(\lambda_1 +\lambda_2 +\lambda_3)\right)^2 +
    \frac{1}{4}(\omega_1 -\omega_2)^2
    -\frac{1}{12}(\lambda_2-\lambda_3)^2, \]
     which is nonnegative, because when $\lambda_2 = \overline{\lambda_3}$, the term $ -\frac{1}{12}(\lambda_2-\lambda_3)^2$ is nonnegative.
     Therefore, we have a similar proof here.

\item[\ref{enum:cds4}.]
     Suppose that the list $\Lambda$ is realizable as the eigenvalues of a doubly stochastic nonnegative matrix. Since the doubly stochastic nonnegative matrix is also a nonnegative matrix, by statement \ref{enum:cg4} of Theorem \ref{the:cg}, $\Lambda$ satisfies conditions $(i)$--$(iii)$ of statement \ref{enum:cds4}.

    Conversely, we suppose that $\Lambda$ satisfy conditions $(i)$--$(iii)$ of statement \ref{enum:cds4}. Then, we may take values $\omega_1, \omega_2$, and $\omega_3$ as in ranges (\ref{eq:cdsw1}) and (\ref{eq:cdsw2}). Then, by statement \ref{enum:cds3}, the list $\Omega$ satisfies the properties $(i)$--$(iv)$ of statement \ref{enum:cds1}. For example let $\omega_1, \omega_2$, and $\omega_3$ be all equal to $\frac{1}{3}(\lambda_1 +\lambda_2+\lambda_3)$. Then, $\omega_3-\frac{\lambda_2}{2}-\frac{\lambda_3}{2}=
    \frac{1}{6}(2 \lambda_1-\lambda_2-\lambda_3)$ and
    $(\lambda_1-\omega_1)(\lambda_1-\omega_2) -\frac{1}{3}(\lambda_1-\lambda_2)(\lambda_1-\lambda_3)
    = \frac{1}{9}(\lambda_1^2 +\lambda_2^2 +\lambda_3^2 - \lambda_1 \lambda_2 -\lambda_1 \lambda_3 - \lambda_2 \lambda_3)$.
    Both are nonnegative, hence by statement \ref{enum:cds1}, $\Lambda$ is realizable as the eigenvalues of a doubly stochastic nonnegative matrix. This proves statement \ref{enum:cds4}.

\item[\ref{enum:cds3}.] We assume that statement \ref{enum:cds1} holds. Then, from statement \ref{enum:cds4}, we have that $\lambda_1 \ge 0$ and $\lambda_2 =  \overline{\lambda_3}$.
    We first show that if the number $\omega_1$ is within range (\ref{eq:cdsw1}), then taking numbers $\omega_2$ and $\omega_3$ with expression (\ref{eq:cdsw2}), the list $\Omega$ satisfies conditions $(i)$--$(iv)$. Second, we show that when $\omega_1$ out of range (\ref{eq:cdsw1}), $\Omega$ does not satisfy them.

    Suppose that $\Omega$ is within ranges (\ref{eq:cdsw1}) and (\ref{eq:cdsw2}).
    From this we can check that $\omega_1+\omega_2 +\omega_3 = \lambda_1 + \lambda_2+ \lambda_3$, which is condition $(ii)$.

    A direct computation shows $\omega_1 - \omega_2 = \frac{1}{2}(3 \omega_1 -\lambda_1 - \lambda_2- \lambda_3)$, which is nonnegative, and $\omega_3$ is nonnegative because
    \begin{multline*}
     \max\{ 0, \frac{1}{4}(2\lambda_2+\lambda_2 +\lambda_3)- \frac{1}{4\sqrt{3}} \sqrt{4(\lambda_1-\lambda_2)(\lambda_1 -\lambda_3)+3(\lambda_2 -\lambda_3)^2} \} \\
    \le \omega_3 \le \frac{1}{3} (\lambda_1 + \lambda_2+\lambda_3).
    \end{multline*}
Hence, condition $(i)$ is satisfied.

A direct computation shows
\begin{equation*}
     \max\{ - \frac{1}{2}(\lambda_2 +\lambda_3), \frac{1}{2}( \lambda_1 - U_2) \}
    \le \omega_3 - \frac{\lambda_2}{2}- \frac{\lambda_3}{2} \le \frac{1}{6} (2\lambda_1 - \lambda_2-\lambda_3).
    \end{equation*}
Now, we show $\frac{1}{2}( \lambda_1 - U_2) \ge 0$, then we have condition $(iii)$.
Consider
\begin{equation*}
\frac{1}{2}( \lambda_1 - U_2)= \\
\frac{1}{4}(2\lambda_2-\lambda_2 -\lambda_3)- \frac{1}{2\sqrt{3}} \sqrt{4(\lambda_1-\lambda_2)(\lambda_1 -\lambda_3)+3(\lambda_2 -\lambda_3)^2}.
\end{equation*}
Since $\frac{1}{4}(2\lambda_2-\lambda_2 -\lambda_3)$ is nonnegative, we may compute the following way
\begin{multline*}
\left(\frac{1}{4}(2\lambda_2-\lambda_2 -\lambda_3)\right)^2
-\left(\frac{1}{2\sqrt{3}} \sqrt{4(\lambda_1-\lambda_2)(\lambda_1 -\lambda_3)+3(\lambda_2 -\lambda_3)^2}\right)^2 \\
=\frac{1}{6}(\lambda_1 -\lambda_2)(\lambda_1 -\lambda_3)
= \frac{1}{6}(\lambda_1 -\lambda_2)(\overline{\lambda_1 -\lambda_2})
\ge 0.
\end{multline*}
Therefore, we have $\frac{1}{2}( \lambda_1 - U_2) \ge 0$.

A direct computation shows
\begin{multline*}
\max\{- \frac{1}{3}(\lambda_1^2 + 2\lambda_1 \lambda_2 + 2\lambda_1 \lambda_3 +  \lambda_2 \lambda_3), 0 \}
\le (\lambda_1-\omega_1)(\lambda_1-\omega_2) -\frac{1}{3}(\lambda_1-\lambda_2)(\lambda_1-\lambda_3) .
\end{multline*}
Hence, condition $(iv)$ holds.
Therefore, $\Omega$ satisfies conditions $(i)$--$(iv)$.

Second, we need to show that if the number $\omega_1$ is out of range (\ref{eq:cdsw1}), then there exists no $\Omega$ that satisfies conditions $(i)$--$(iv)$.
For this proof, we refer to the proof of statement \ref{enum:rds3} of Theorem \ref{the:rds}.
Therefore, we conclude statement \ref{enum:cds3}.
\end{enumerate}
\end{proof}

From conditions $(ii)$--$(iv)$ of statement \ref{enum:cds1} of Theorem \ref{the:cds}, we may write $\lambda_1=a, \lambda_2 =b+ci$ and $\lambda_3=b-ci$ where $a, b$, and $c$ are real numbers. We restate statements \ref{enum:cds3} and \ref{enum:cds4} of Theorem \ref{the:cds} in terms of $a, b$, and $c$. The list becomes $\Lambda = \{ a, b+c i, b-c i \}.$

\begin{corollary}
\begin{enumerate}
\item The list $\Lambda$ is realizable as the eigenvalues of a doubly stochastic nonnegative matrix if and only if the following hold:
    \begin{description}
        \item[] (i) $a \ge 0$,
        \item[] (ii) $-\frac{a}{2}  \le b \le a$,
        \item[] (iii) $(a-b)^2 \ge 3c^2 $.
    \end{description}

\item  The exact range for $\omega_1$ is
    \[
    \frac{1}{3}(a+2b) \le \omega_1 \le \min\{a+2b, b+\frac{1}{\sqrt{3}}\sqrt{(a-b)^2-2c^2}\}.
    \]
\end{enumerate}
\end{corollary}

\bibliographystyle{abbrv}

\end{document}